\documentclass[12pt]{article}
\usepackage{layout,yfonts, graphicx, amsmath, amsthm, amssymb, euscript, enumerate, bbold, bbm, mathrsfs}

\usepackage[numbers]{natbib} 
\usepackage{hyperref}
\numberwithin{equation}{section}
\usepackage{color}
\definecolor{webgreen}{rgb}{0,.5,0}
\definecolor{webbrown}{rgb}{.8,0,0}
\definecolor{emphcolor}{rgb}{0.95,0.95,0.95}

\usepackage{hyperref}
\hypersetup{%
	colorlinks=true,
	linkcolor=webbrown,
	filecolor=webbrown,
	citecolor=webgreen,
	breaklinks=true}
\ifpdf \hypersetup{pdftex,
	bookmarksopen=true,
	bookmarksnumbered=true
} \else \hypersetup{dvips} \fi

\usepackage{blindtext} 
\usepackage{scrextend}
\addtokomafont{labelinglabel}{\sffamily}

\usepackage[colorinlistoftodos, textwidth=21mm, shadow,textsize=footnotesize]{todonotes}
\reversemarginpar
\usepackage[margin=1in]{geometry}
\theoremstyle{plain}
  \newtheorem{teor}{Theorem}[section]
  \newtheorem{defin}[teor]{Definition}
  \newtheorem{prop}[teor]{Proposition}
  \newtheorem{lema}[teor]{Lemma}
  \newtheorem{coro}[teor]{Corollary}
\theoremstyle{remark}  
  \newtheorem{rem}[teor]{Remark}

\newtheoremstyle{hyp}{}{}{\itshape}{}{}{}{3pt}{}
\theoremstyle{hyp}

\DeclareMathAlphabet{\mathpzc}{OT1}{pzc}{m}{it}
\DeclareMathOperator*{\limess}{lim\, inf\, ess}
\DeclareMathOperator{\deri}{D}
\DeclareMathOperator{\dist}{dist}

\DeclareMathOperator{\Lp}{L}
\DeclareMathOperator{\trans}{T}
\DeclareMathOperator{\comp}{c} 
\DeclareMathOperator{\loc}{loc}
\DeclareMathOperator{\hol}{C}
\DeclareMathOperator{\tr}{tr}
\DeclareMathOperator{\sob}{W}

\DeclareMathOperator{\expo}{e}
\DeclareMathOperator{\sop}{supp}
\DeclareMathOperator{\inted}{\mathcal{I}}

\newcommand{\E}{\mathbbm{E}}
\newcommand{\set}{\mathcal{O}}
\newcommand{\dif}{\mathcal{L}}
\newcommand{\inte}{\mathcal{I}}

\newcommand{\ind}{\mathcal{D}}

\newcommand{\R}{\mathbbm{R}}

\newcommand{\uno}{\mathbbm{1}}
\newcommand{\der}{\mathrm{d}}
\newcommand{\F}{\mathbbm{F}}
\newcommand{\Pro}{\mathbbm{P}}
\newcommand{\BE}{\begin{equation}}

\newcommand{\EE}{\end{equation}}
\newcommand {\BA}{\begin{align}}
\newcommand{\EA}{\end{align}}
\newcommand{\eqdef}{\raisebox{0.4pt}{\ensuremath{:}}\hspace*{-1mm}=}
\newcommand{\defeq}{=\hspace*{-1mm}\raisebox{0.4pt}{\ensuremath{:}}}

\title{\Large{\bf HJB equations with gradient constraint associated with controlled jump-diffusion  processes}\footnote{\textbf{Funding:}  {This} study has been funded by the Russian Academic Excellence Project `5-100'.}}
\author{\large{\bf Mark Kelbert}\\ \small{\it Laboratory of Stochastic Analysis and its Applications}\\
\small{\it National Research University Higher School of Economics, Moscow, Russia}\\
\large{\bf Harold A. Moreno-Franco}\footnote{Corresponding author: hamoreno@uninorte.edu.co}\\
\small{\it Department  of Mathematics and Statistics}\\
\small{\it Universidad del Norte, Barranquilla, Colombia}\\
\small{\it and}\\
\small{\it Laboratory of Stochastic Analysis and its Applications}\\ 
\small{\it National Research University Higher School of Economics, Moscow, Russia}}
\date{}
\begin{document}	
\maketitle
\vspace{-0.5cm}
\begin{abstract}
\noindent 
In this paper, we guarantee  {the} existence and uniqueness (in  {the}  almost everywhere sense) of the solution to  a Hamilton-Jacobi-Bellman  (HJB) equation with gradient constraint and a partial integro-differential operator whose L\'evy measure has bounded variation. This type of equation arises in a singular control problem,  where the state process is a  {multidimensional} jump-diffusion with jumps of finite variation and infinite activity.  {We verify, by means of $\varepsilon$-penalized controls,  that the value function associated with this problem satisfies the aforementioned HJB equation.} 
\end{abstract}

\section{Introduction}

 {Our main goal is to study} the following HJB equation, 
\begin{equation}\label{p1}	
\max\{\Gamma u-h,|\deri^{1} u|-g\}=0,\  \text{a.e. in}\ \set,\quad \text{s.t.}\  u=0,\ \text{on}\  \overline{\set}_{\inted}\setminus\set,
\end{equation}
where   $\set$ is a convex, open and  bounded set  such that  $\set\subset\overline{\set}_{\inted}\subseteq \R^{d}$  and its boundary $\partial\set$ is of class $\hol^{3,\alpha'}$, with $\alpha'\in(0,1)$ fixed.   The set $\overline{\set}_{\inted}$ shall be given later on. The partial integro-differential operator  $\Gamma$ is defined by
\begin{align}\label{p6}
  \Gamma u(x)&=\dif u(x)-\inte u(x),
\end{align}  
with  
\begin{align}\label{p6.1}
\begin{split}
\dif u(x)&\eqdef -\tr[a(x)\deri^{2} u(x)]+\langle b(x),\deri^{1} u(x)\rangle+c(x)u(x),\\
\inte u(x)&\eqdef\int_{\R^{d}_{*}}[u(x+z)-u(x)]s(x,z)\nu(\der z),
\end{split} \qquad \text{for}\ x\in\set.
\end{align}
Here $|\cdot|$, $\langle\cdot,\cdot\rangle$ and $\tr[\,\cdot\,]$  {represent} the Euclidean norm, the inner product, and  the  {matrix trace}, respectively; $\deri^{1}u=(\partial_{1}u,\dots,\partial_{d}u)$, $\deri^{2}u=(\partial_{ij}u)_{d\times d}$,
 $h,c:\overline{\set}\longrightarrow\R$,   $g:\overline{\set}_{\inted}\longrightarrow\R$, $b:\overline{\set}\longrightarrow\R^{d}$, $a:\overline{\set}\longrightarrow\mathcal{S}(d)$,    with $\mathcal{S}(d)$ the set of $d\times d$ symmetric matrices,  $\nu$ is a Radon measure on $\R^{d}_{*}\eqdef\R^{d}\setminus\{0\}$  {satisfying}
 \begin{equation}\label{H4}
 \int_{\R^{d}_{*}}[|z|\wedge 1]\nu(\der z)\leq C_{\nu}, 
 \end{equation}
 for some finite positive constant $C_{\nu}$, and  $s:\overline{\set}\times\R^{d}\longrightarrow[0,1]$ is such that
\begin{equation}\label{p2}
\int_{\R^{d}_{*}}s(x,z)\uno_{\{x+z\notin\set\}}\nu(\der z)<\infty,\  \text{for}\ x\in\set. 
\end{equation}

 {The} notations concerning  function spaces  { that we have}  used in the paper  are standard   and     are discussed in Subsection \ref{notation}.

The HJB equation \eqref{p1} when $\Gamma=\dif$  was introduced by  Evans in 1979 \cite{evans}. {Under some regularity  assumptions on the coefficients of \eqref{p1}, and $\dif$ satisfying the elliptic property},  he {showed} that the unique solution to this problem belongs to  $\sob^{1,\infty}(\set)\cap\sob^{2,p}_{\loc}(\set)$,  for each $p\in[1,\infty)$. Shortly  {afterwards,} Wiegner \cite{wieg} proved that this solution is in $\hol^{1,1}(\set)$.  Later on, Ishii and  Koike \cite{ishii}  considered this problem with a gradient constraint more general  than  Evans proposed in \cite{evans}. They verified that the solution to their HJB equation is in $\sob^{2,\infty}(\set)$.   {Then,}   Hynd \cite{hynd2} studied the problem with a convex gradient constraint and showed that the solution to this problem is in a viscosity sense and belongs to  $\hol^{1,\alpha}_{\loc}(\set)\cap\hol^{0,1}(\set)$, for  $\alpha\in(0,1)$.

Recently, Moreno-Franco \cite{moreno} analysed the HJB equation \eqref{p1}  when the domain set is a ball $B_{R}(0)\subset\
\R^{d}$,  the coefficients of the partial integro-differential operator $\Gamma$ are constant, $s=g=1$, $c=q$, with $q$ {being} a positive constant large enough, and the L\'evy measure $\nu$ has a density $\kappa\in\hol^{0,\alpha'}(\R^{d}_{*})$ with respect to the Lebesgue measure  $\der z$ such that $\nu(\R^{d}_{*})<\infty$. In this case, assuming that $h\in\hol^{2}(\overline{B_{R}(0)})$ is non-negative, $\nu$ is such that $\int_{\R^{d}_{*}}|z|\nu(\der z)<\infty$, and using  {PDEs} and probabilistic methods, the author proved the equation \eqref{p1} has a unique solution  in $\hol^{0,1}(\overline{B_{R}(0)})\cap\sob^{2,p}_{\loc}(B_{R}(0))$ a.e., for each $p\in(d,\infty)$.  {It was also} shown that there is a relationship  between the HJB equation \eqref{p1} on the whole space $\R^{d}$ and a singular control problem, when the controlled process is a L\'evy process, whose components are a $d$-dimensional standard Brownian motion (SBM) with  drift and a Poisson compound process. 

Notice that the HJB equation \eqref{p1} with the operator $\Gamma$ defined as in \eqref{p6} is more general than in \cite{moreno}. The L\'evy measure $\nu$ can  satisfy $ \nu(\R^{d}_{*})=\infty$ and it is not required that $\nu$ has a density $\kappa$ with respect to the Lebesgue measure $\der z$.    {This} type of HJB equation is also related to a singular control problem  when the state process is a jump-diffusion process $X=\{X_{\mathpzc{t}}:\mathpzc{t}\geq0\}$ (see Eq. \eqref{esd1}) with  infinitesimal generator of the form
\begin{equation}\label{p.13.1}
\tr[a\deri^{2} u]-\langle b,\deri^{1} u\rangle+\int_{\R^{d}_{*}}[u(\cdot+z)-u]s(\cdot,z)\nu(\der z),\ \text{on}\  \overline{\set}.
\end{equation}
The last term in \eqref{p.13.1} corresponds  to the infinitesimal generator of a jump process, whose \textit{jump size and rate} are given by $z\in\R^{d}_{*}$, and $s(x,z)$, respectively. The jump rate of the process $X$ depends on its position at time $t$. For more detail about this problem, see Subsection \ref{Prob.1}.

\subsection*{Assumptions and main results}\label{ass1}
 The following assumptions will  {henceforth} be  imposed:

\begin{enumerate}
	\item[(A1)]Assume that    $h,c,a_{ij},b_{i},c\in\hol^{1,\alpha'}(\overline{\set})$, with $\alpha'\in(0,1)$ fixed, $g\in\hol^{2}(\overline{\set})\cap\hol^{1}(\overline{\set}_{\inted})$ and $||h||_{\hol^{1,\alpha'}(\overline{\set})}$, $ ||a_{ij}||_{\hol^{1,\alpha'}(\overline{\set})}$, $||b_{i}||_{\hol^{1,\alpha'}(\overline{\set})}$, $||c||_{\hol^{1,\alpha'}(\overline{\set})}$, $||g||_{\hol^{2}(\overline{\set})}$ and $||g||_{\hol^{1}(\overline{\set}_{\inted})}$     are bounded by some finite positive constant $\Lambda$. 
	
	\item[(A2)] The functions $h$, $g$ and $c$ are such that $h\geq0$, $c>0$  on $\overline{\set}$, and  $g\geq0$ on $\overline{\set}_{\inted}$. 
	
	\item[(A3)] The differential part of the operator $\Gamma$ is   \textit{strictly elliptic}; i.e., there exists a real number $\theta>0$  such that $\langle a(x) \zeta,\zeta\rangle\geq \theta|\zeta|^{2},\ \text{for all}\  x\in\overline{\set},\ \zeta\in \R^{d}$.
	
	\item[(A4)]  Finally, we assume that $\nu$ is a Radon measure on $\R^{d}_{*}$  satisfying \eqref{H4} and  {$s\in \hol^{1,\alpha'}(\overline{\set}\times\R^{d})$ is such that \eqref{p2} holds}.   
\end{enumerate}

 {Before introducing the main results of the paper, let us define the support $\overline{\set}_{\inted}$ of the operator $\inted$. Consider the L\'evy kernel $M_{s}(x,B)=\int_{z\in B}s(x,z)\nu(\der z)$, where $x\in\set$ and $B$ a Borel measurable set of $\R^{d}_{*}$. Then,} 
\begin{align}\label{p3}
\overline{\set}_{\inted}\eqdef\overline{\bigcup_{x\in \overline{\set}}\{x+[\R^{d}\setminus \mathcal{Z}_{\inted}(x)]\}}, 
\end{align} 
where   $\mathcal{Z}_{\inted}(x)\eqdef\{z'\in\R^{d}_{*}:M_{s}(x,B_{\epsilon}(z'))=0,\ \text{for some}\ \epsilon\in(0,|z'|)\}$; see \cite[Definition 2.3.10]{garroni}. 
The set $\mathcal{Z}_{\inted}(x)$ is called the {\it zero-jump set}. Notice that $ \overline{\set} \subset \overline{\set}_{\inted}$ and $\overline{\set}= \overline{\set}_{\inted}$ if $s(x,z)=0$, for all $(x,z)\in \overline{\set} \times\R^{d}_{*}$ such that $x+z\notin\overline{\set}$.

Without loss of generality consider $\overline{\set}\subset\overline{\set}_{\inted}\subset\R^{d}$ from now on. Taking the operator $\Gamma$ as in \eqref{p6} and under  {Assumptions} (A1)--(A4), the main {goal} obtained in this document  is  as follows:
\begin{teor}\label{princ1.0.1}
	For each $p\in(d,\infty)$, there exists a unique non-negative solution $u$ to the HJB equation \eqref{p1} in the space  $\hol^{0,1}(\overline{\set})\cap\sob^{2,p}_{\loc}(\set)$.
\end{teor}
The solution $u$ to the HJB equation \eqref{p1} is established in  {the} almost everywhere sense in line with \cite{moreno}. To prove  Theorem \ref{princ1.0.1}; see Section \ref{pp2},  we will employ a penalization technique, which  has been used  {by} \cite{evans, hynd, hynd2, Hynd3, ishii, moreno, soner, wieg},  when the operator $\Gamma$ has only the elliptic differential part $\dif$ or when the L\'evy measure of its integral part $\inted$ is finite.  Considering the non-linear partial integro-differential Dirichlet (NPIDD) problem 
\begin{equation}\label{p13.0}
\Gamma  u^{\varepsilon}+ \psi_{\varepsilon}(|\deri^{1} u^{\varepsilon}|^{2}- g^{2})=h,\ \text{in}\ \set,\ \quad\text{s.t.}\ u^{\varepsilon}=0,\ \text{on}\ \overline{\set}_{\inted}\setminus\set  ,
\end{equation}
where the \textit{penalizing function} $\psi_{\varepsilon}:\R\longrightarrow\R$, with $\varepsilon\in(0,1)$, belongs to $\hol^{\infty}(\R)$ and is determined as
\begin{equation}\label{p12.1}
\begin{split}
\psi_{\varepsilon}(r)&=0,\ r\leq0,\quad 
\psi_{\varepsilon}(r)>0,\ r>0,\\ 
\psi_{\varepsilon}(r)&=\frac{r-\varepsilon}{\varepsilon},\  r\geq2\varepsilon,\quad
\psi_{\varepsilon}'(r)\geq0,\quad \psi_{\varepsilon}''(r)\geq0,
\end{split}
\end{equation}
we first guarantee  {the} existence and uniqueness of the classical  solution  {$u^{\varepsilon}$} to the NPIDD problem \eqref{p13.0}, with  $\Gamma$ as in \eqref{p6}. Once this is done, we establish  uniform estimates of  {the sequence $\{u^{\varepsilon}\}̣_{\varepsilon\in(0,1)}$} that allow us to pass to  the limit as $\varepsilon\rightarrow0$,  in a weak sense in \eqref{p13.0}, which leads to the existence and regularity of the solution to the HJB equation  \eqref{p1}. 

Under Assumptions (A1)--(A4), the other main result obtained in the paper is as follows:

\begin{prop}\label{princ1.0}
	For each $\varepsilon\in(0,1)$, there exists a unique non-negative solution $u^{\varepsilon}$ to the NPIDD problem  \eqref{p13.0} in the space $\hol^{3,\alpha'}(\overline{\set})$.
\end{prop}
 Although the NPIDD problem \eqref{p13.0} is a tool to guarantee the existence of the  solution to the HJB equation \eqref{p1}, this turns out to be a problem of interest itself  because, previously to this paper, we find few references related to this class of problems. Say, paper \cite{moreno} analyses  the NPIDD problem \eqref{p13.0} when the L\'evy measure $\nu$ is finite on $\R^{d}_{*}$, and \cite{taira} studies a degenerate Neumann problem for quasi-linear elliptic integro-differential operators when  the L\'evy measure $\nu$ has unbounded variation, i.e., $\int_{\R^{d}_{*}}[|z|^{2}\wedge1]\nu(\der z)<\infty$,  and $s$ satisfies $s(x,z)=0$, for $(x,z)\in\overline{\set}\times\R^{d}_{*}$ such that  $x+z\notin\overline{\set}$.  {This} type of problem  can  {also} be  related to an absolutely continuous optimal control problem when the controlled process is a jump-diffusion with jump measure of finite variation; see Section \ref{HJBPro}.

To finalize this part, let us make  some comments about the assumptions mentioned {in the beginning of this subsection}. Under (A1), (A3), (A4) and the fact that the boundary $\partial\set $ is of class $\hol^{3,\alpha'}$, we ensure the existence and uniqueness of the classical solution $\mathpzc{u}$ to the linear partial integro-differential Dirichlet (LPIDD) problem in \eqref{p13.0.0} when $w\in\hol^{1,\alpha'}(\overline{\set})$; see \cite[Thm. 3.1.12]{garroni}.  Assumptions (A1),  (A2), (A4) and that $\set$ is a bounded convex set are required to show some \textit{\`a priori} estimates of the solution $u^{\varepsilon}$ to the NPIDD problem \eqref{p13.0}, which must be independent of $\varepsilon$; see Lemmas \ref{Lb1}--\ref{gr1}. Since $h\geq0$, $c>0$ on $\overline{\set}$ and using Lemma \ref{supsuper}, it can be verified that $u^{\varepsilon}$ is the unique non-negative solution to the NPIDD problem \eqref{p13.0}; see Subsection \ref{pp1}. Finally, {once again making use of} $c>0$ on $\overline{\set}$, it is  {proven} that the solution to the HJB equation \eqref{p1} is unique; see Subsection \ref{proofHJB1}.  

 {The rest of this} document is organized as follows: Section \ref{apr1} is devoted to  {prove}  the existence and uniqueness  of  the solution  to the NPIDD problem \eqref{p13.0}.   {First}, some properties of the integral operator $\inte$ and some \textit{\`a priori} estimates of the solution to the NPIDD problem \eqref{p13.0} are studied.  Afterwards,  using Lemmas \ref{g1}, \ref{Lb1}, \ref{gr1}, \ref{gr2}, and the Schaefer fixed point Theorem; see \cite[Thm. 4, p. 539]{evans2},  it is  {proven} that the classical solution $u^{\varepsilon}$ to the NPIDD problem \eqref{p13.0} exists and is unique; see Subsection \ref{pp1}. Then, in Section \ref{pp2}, by Lemmas \ref{Lb1}, \ref{gr1}, \ref{cotaphi}, \ref{lemafrontera4},  using Arzel\`a-Ascoli Theorem and the reflexivity of $\Lp^{p}_{\loc}(\set)$; see \cite[Thm. 7.25, p. 158 and Thm. 2.46, p. 49, respectively]{rudin, adams}, we extract a convergent sub-sequence of $\{u^{\varepsilon}\}_{\varepsilon\in(0,1)}$, whose limit is the solution to the HJB equation \eqref{p1}; see Subsection \ref{proofHJB1}.  {In the following subsection, we present } the singular control problem that is related to the HJB equation \eqref{p1}. The probabilistic arguments of this part are given in  Section \ref{HJBPro}.  Finally, we draw  {our}  conclusions and discuss possible extensions of this paper.

\subsection{Probabilistic interpretation}\label{Prob.1} 
Let $W=\{W_{\mathpzc{t}}:\mathpzc{t}\geq0\}$ and $N$ be a $d$-dimensional SBM and a Poisson random measure on $(\mathcal{S}\times[0,\infty),\mathcal{B}(\mathcal{S})\times\mathcal{B}([0,\infty)),\eta(\der\rho,\der z)\times\der \mathpzc{t})$, with $\mathcal{S}\eqdef [0,1]\times\R^{d}$ and $\eta(\der\rho,\der z)=\der\rho\nu(\der z)$, respectively, which are  defined on a complete probability space  $(\Omega,\mathcal{F},\Pro)$. We assume that $W$ and $N$ are independent. Let $\F=\{\mathcal{F}_{\mathpzc{t}}\}_{\mathpzc{t}\geq0}$ be the filtration  generated by $W$ and $N$.  
We assume furthermore that the filtration $\F$ is completed with the null sets of $\Pro$. The uncontrolled stochastic process $X=\{X_{\mathpzc{t}}:\mathpzc{t}\geq0\}$ is governed by the stochastic differential equation (SDE)
\begin{equation}\label{esd1}
X_{\mathpzc{t}}=\tilde{x}-\int_{0}^{\mathpzc{t}} \tilde{b} (X_{\mathpzc{s}})\der \mathpzc{s}+\int_{0}^{\mathpzc{t}}\sigma(X_{\mathpzc{s}})\der  W_{\mathpzc{s}}+\int_{0}^{\mathpzc{t}}\der  J_{\mathpzc{s}},\  \mathpzc{t}>0,
\end{equation}
where $\tilde{x}\in\set$, $ \tilde{b} :\R^{d}\longrightarrow\R^{d}$, and   $\sigma:\R^{d}\longrightarrow\R^{d\times d}$. The jump process $ J $ is defined  by
\begin{align}\label{esd2.1}
J_{\mathpzc{t}}&=\int_{0}^{\mathpzc{t}}\int_{\mathcal{S}_{1}}z\uno_{\{\rho\in[0,  s (X_{\mathpzc{s}-},z)]\}} \widetilde{N}(\der \rho,\der z,\der \mathpzc{s})+\int_{0}^{\mathpzc{t}}\int_{\mathcal{S}\setminus\mathcal{S}_{1}}z\uno_{\{\rho\in[0,  s (X_{\mathpzc{s}-},z)]\}} N(\der \rho,\der z,\der \mathpzc{s}),
\end{align}
with $s:\R^{d}\times\R^{d}\longrightarrow[0,1]$, $\mathcal{S}_{1}\eqdef\{(\rho,z)\in\mathcal{S}:|z|\in(0,1)\}$, and $\widetilde{N}(\der \rho,\der z,\der \mathpzc{t})\eqdef N(\der \rho,\der z,\der \mathpzc{t})-\eta(\der \rho,\der z)\der \mathpzc{t}$ is the  compensated Poisson random  measure with intensity $\eta(\der\rho,\der z)\der\mathpzc{t}$. For each $\tilde{x}\in\set$, $\Pro_{\tilde{x}}$ represents the probability law of $X$ when it starts at $\tilde{x}$, and $\E_{\tilde{x}}$ is the expected value associated with $\Pro_{\tilde{x}}$.

In addition to  (A1)--(A4), we need to add other assumption on the whole space $\R^{d}$ in such a way that the SDE \eqref{esd1} has a unique c\`adl\`ag adapted solution $X$. This assumption will only be  used here and in Section \ref{HJBPro}. 
\begin{enumerate}
		\item[(A5)] Assume that there exists a positive constant $C$ such that
		\begin{equation}\label{jum.1}
		\begin{split}
		&|\sigma(x)|^{2}+|\tilde{b}(x)|^{2}\leq C[1+|x|^{2}],\\
		&|\sigma(x)-\sigma(y)| +|\tilde{b}(x)-\tilde{b}(y)|\leq C|x-y|\\
		&\int_{\{|z|\in(0,1)\}}|z|\,|s(x,z)-s(y,z)|\nu(\der z)\leq C|x-y|,
		\end{split}
		\end{equation} 
		for $x,y\in\R^{d}$ with $x\neq y$.
	\end{enumerate}
	\begin{rem}
		Notice that for each $x,y\in\R^{d}$ with $x\neq y$, 
		\begin{align}\label{jum.2}
		\begin{split}
		\int_{\mathcal{S}_{1}}|z|^{2}\uno_{\{0\leq\rho\leq s(x,z)\}}\eta(\der\rho,\der z)&=\int_{\{|z|\in(0,1)\}}|z|^{2}s(x,z)\nu(\der z)\leq C_{\nu},\\
		\int_{\mathcal{S}_{1}}|z|\,|\uno_{\{0\leq\rho\leq s(x,z)\}}-\uno_{\{0\leq\rho\leq s(y,z)\}}|\eta(\der\rho,\der z)&=\int_{\{|z|\in(0,1)\}}|z|\,|s(x,z)-s(y,z)|\nu(\der z)\\
		&\leq C|x-y|,
		\end{split}
		\end{align}
		since (A4) holds. Then, from \eqref{jum.1}--\eqref{jum.2},  the SDE \eqref{esd1} has a unique c\`adl\`ag adapted solution $X$; see \cite{KK2014}.
	\end{rem}

Since $\int_{\{|z|\in(0,1)\}}|z|\nu(\der z)<\infty$  and $\eta(\der\rho,\der z)=\der \rho\nu(\der z)$, the infinitesimal generator of $X$ is given by
\begin{align}\label{esd3.0}
\Gamma_{1} u(x)&=\tr[a(x)\deri^{2}u(x)]-\langle  \tilde{b}(x) ,\deri^{1}u(x)\rangle\notag\\
&\quad+\int_{\mathcal{S}}[u(x+z\uno_{\{\rho\in[0,s(x,z)]\}})-u(x)-\langle\deri^{1}u(x),z\rangle\uno_{\{\rho\in[0, s(x,z)],\,|z|\in(0,1)\}}]\eta(\der\rho,\der z)\notag\\
&=\tr[a(x)\deri^{2}u(x)]-\langle  b(x) ,\deri^{1}u(x)\rangle+\int_{\R^{d}_{*}}[u(x+z)-u(x)]   s   (x,z)\nu(\der z), 
\end{align}
where  $a_{ij}=\frac{1}{2}(\sigma\sigma^{\trans})_{ij}$ and $ b =  \tilde{b} +\int_{\{ |z|\in(0,1)\}}z s   (\cdot,z)\nu(\der z)$. Let $\mathcal{U}$ be the admissible class of control processes $(n,\zeta)$ that satisfies
\begin{equation}\label{cont.1}
\begin{cases}
(n_{\mathpzc{t}},\zeta_{\mathpzc{t}})\in\R^{d}\times\R_+,\ \mathpzc{t}\geq0,\ 
(n,\zeta)\ \text{is adapted to the filtration}\  \F,\ \zeta_{0-}=0\\
\text{and}\ \zeta_{\mathpzc{t}}\ \text{is non-decreasing and is right continuous with left hand limits,}\ \mathpzc{t}\geq0,\\
\text{and }\ |n_{\mathpzc{t}}|=1\ {\der\zeta_{\mathpzc{t}}\text{-a.s.},\ \mathpzc{t}\geq0} .
\end{cases}
\end{equation}
Then, for each $(n,\zeta)\in\mathcal{U}$ and $\tilde{x}\in\set$, the process $X^{n,\zeta}=\{X^{n,\zeta}_{\mathpzc{t}}:\mathpzc{t}\geq0\}$ evolves as
\begin{align}\label{esd3.1}
X^{n,\zeta}_{\mathpzc{t}}=\tilde{x}-\int_{0}^{\mathpzc{t}} \tilde{b} (X^{n,\zeta}_{\mathpzc{s}})\der \mathpzc{s}+\int_{0}^{\mathpzc{t}}\sigma(X^{n,\zeta}_{\mathpzc{s}})\der  W_{\mathpzc{s}}+\int_{0}^{\mathpzc{t}}\der  J_{\mathpzc{s}}-\int_{[0,\mathpzc{t}]}n_{\mathpzc{s}}\der \zeta_{\mathpzc{s}},\ \mathpzc{t}\geq0.
\end{align}
The process $n$ provides the direction and $\zeta$  the intensity of the push applied to the state process $X^{n,\zeta}$.   Since \eqref{jum.1}--\eqref{jum.2} hold, we get that the SDE \eqref{esd3.1} has a unique c\`adl\`ag adapted solution $X^{n,\zeta}$; see \cite{dade}. The jumps of $X^{n,\zeta}$ are given by  the processes $J$ and $\zeta$, i.e.,
$\Delta X^{n,\zeta}_{\mathpzc{t}}\eqdef X^{n,\zeta}_{\mathpzc{t}}-X^{n,\zeta}_{\mathpzc{t}-}=\Delta  J_{\mathpzc{t}}-n_{\mathpzc{t}}\Delta \zeta_{\mathpzc{t}}$, for $\mathpzc{t}\geq0$. The \textit{cost function} corresponding to  $(n,\zeta)\in\mathcal{U}$, is defined as
\begin{align}\label{esd1.1}
V_{n,\zeta}(\tilde{x})&=\E_{\tilde{x}}\biggr[\int_{[0,\tau^{n,\zeta}]}\expo^{-q\mathpzc{t}}[ h (X^{n,\zeta}_{\mathpzc{t}})\der \mathpzc{t}+ g   (X^{n,\zeta}_{\mathpzc{t}-})\circ\der\zeta_{\mathpzc{t}}]\biggr],\ \tilde{x}\in\overline{\set},
\end{align}
where  $\tau^{n,\zeta}\eqdef\inf\{ \mathpzc{t}>0:X_{\mathpzc{t}}^{n,\zeta}\notin\set\}$, $q$ is a positive constant and
\begin{multline}\label{esd.1}
\int_{[0,\mathpzc{t}]}\expo^{-q\mathpzc{s}} g   (X^{n,\zeta}_{\mathpzc{s}-})\circ\der\zeta_{\mathpzc{s}}\eqdef\int_{0}^{\mathpzc{t}}\expo^{-q\mathpzc{s}} g   (X^{n,\zeta}_{\mathpzc{s}})\der\zeta_{\mathpzc{s}}^{\comp}\\
+\sum_{0\leq\mathpzc{s}\leq\mathpzc{t}}\expo^{-q\mathpzc{s}}\Delta\zeta_{\mathpzc{s}}\int_{0}^{1} g   (X^{n,\zeta}_{\mathpzc{s}-}+\Delta  J_{\mathpzc{s}}-\lambda n_{\mathpzc{s}}\Delta\zeta_{\mathpzc{s}})\der\lambda,\ \text{for}\ \mathpzc{t}>0,
\end{multline}
where $\zeta^{\comp}$ denotes the continuous part of $\zeta$ and $ h , g:\R^{d}\longrightarrow\R   $ are continuous and non-negative. Notice that each control $(n,\zeta)\in\mathcal{U}$ generates two types of costs  {because} $(n,\zeta)$ controls the process $X^{n,\zeta}$  continuously or by jumps of $\zeta$ while $X^{n,\zeta}$ is inside $\set$. The term $\int_{0}^{1} g   (X^{n,\zeta}_{\mathpzc{s}-}+\Delta  J_{\mathpzc{s}}-\lambda n_{\mathpzc{s}}\Delta\zeta_{\mathpzc{s}})\der\lambda$ represents the cost for using the   jump  $\Delta\zeta_{\mathpzc{s}}\neq0$ with direction $-n_{\mathpzc{s}}$ on $X^{n,\zeta}_{\mathpzc{s}-}+\Delta  J_{\mathpzc{s}}$ at time $ \mathpzc{s}$. The \textit{value function} is defined by
\begin{equation}\label{vf1}
V(\tilde{x})=\inf_{(n,\zeta)\in\mathcal{U}}V_{n,\zeta}(\tilde{x}),\ \text{for}\ \tilde{x}\in\overline{\set}.
\end{equation}

A heuristic derivation from dynamic programming principle; see \cite[Ch. VIII]{flem}, shows that the HJB equation corresponding to the value function $V$ is given by
\begin{equation}\label{esd5}
\max\{[q-\Gamma_{1}]u-  h ,|\deri^{1}u|- g   \}= 0,\ \text{on}\ \set,\quad \text{s.t.}\ u=0,\ \text{in}\ \overline{\set}_{\inted}\setminus\set ,
\end{equation}
where $\Gamma_{1}$ is as in \eqref{esd3.0}. An immediate consequence of Theorem \ref{princ1.0.1} is the following corollary. 

\begin{coro}
Assume that $a_{ij}$, $ b _{i}$, $ h $, $ g   $, $  s  $ satisfy (A1)--(A5). Then, the HJB equation \eqref{esd5} has a unique non-negative solution $u$ in $\hol^{0,1}(\overline{\set})\cap\sob^{2,p}_{\loc}(\set)$, for each $p\in(d,\infty)$.
\end{coro}

\begin{prop}\label{veri1}
	Let $u$ be the non-negative solution to the HJB equation \eqref{esd5}. Then,  $V$ defined in \eqref{vf1} and $u$ agree on $\overline{\set}$. 
\end{prop}
To give the  proof of this proposition, we need to introduce a class of penalized controls which are related to the singular control problem described above and the NPIDD problem \eqref{p13.0}. For more detail, see Section \ref{HJBPro}.

\subsubsection*{Comments}
\begin{rem}
Previously to the paper  {by}  Moreno-Franco \cite{moreno} and  {this paper}, the singular stochastic control problem described above has been studied extensively in the   one-dimensional case when the state process  {includes} the continuous part only; see, e.g., \cite{A1999,DZ1998,GZ2015,JJZ2008, K1983}. Several articles  {focused} on the  multidimensional case when  the state process is a multidimensional SBM \cite{evans,kruk,menaldi,soner}, a diffusion process \cite{flem,hynd2, Hynd3},  or  a multidimensional SBM with jumps process, whose L\'evy measure $\nu$ satisfies $\int_{\R^{d}_{*}}|z|^{p}\nu(\der z)<\infty$, for all $p\geq2$ \cite{menal2}.  It should be noted that the results in \cite{A1999,DZ1998,GZ2015,JJZ2008, K1983,kruk, menal2, menaldi,  soner},  {were} given on the whole space $\R^{d}$, and that in \cite{soner}, under convexity and polynomial growth assumptions on the function $ h $, it is  {shown} that the value function associated with a controlled two-dimensional SBM is in $\hol^{2}(\R^{2})$. 
\end{rem}
\begin{rem}	
For the one-dimensional case,  similar problems to ours can be found in the mathematical finance and risk theory; see, e.g.,  \cite{DDKL2016,Y2017} and \cite{app,BKY2014}, respectively. In the risk theory, one wishes to determine an optimal dividend payment strategy for an insurance company (or discovery company) to pay its shareholders, where  {the} insurance company's surplus is modelled by  a spectrally negative (or positive) L\'evy process, i.e., a stochastic process which has a c\`adl\`ag path, and stationary and independent increments without positive (negative) discontinuity. Using some results of fluctuation theory, it can be shown (in some cases) that the value function associated with this problem is in $\hol^{2}(\R)$  and satisfies a similar HJB equation as in \eqref{esd5} on the whole space $\R$; see, e.g., \cite{app,BKY2014}.
\end{rem}
\begin{rem}	
Some ideas given here  and in Section \ref{HJBPro} are taken from \cite{zhu}, where the author has shown that the value function associated with  a controlled multidimensional diffusion process, satisfies the dynamic programming  variational inequality in  {the}  almost everywhere sense. 
\end{rem}
	
\subsection{Notation}\label{notation}
	
We introduce the notation and basic definitions of some spaces that are used in this paper. Let $\alpha\in[0,1]$ and $m\in\{0,\dots,k\}$, with $k\geq0$ an integer. The set $\hol^{k}(\set)$ consists of real-valued functions on $\set$ that are $k$-fold  {continuously} differentiable. We define $\hol^{\infty}(\set)=\bigcap_{k=0}^{\infty}\hol^{k}(\set)$. The sets $\hol^{k}_{\comp}(\set)$  and $\hol^{\infty}_{\comp}(\set)$ consist  of functions in $\hol^{k}(\set)$ and $\hol^{\infty}(\set)$, whose support is compact and contained in $\set$, respectively. The set $\hol^{k}(\overline{\set})$ is defined as the set of real-valued functions such that $\partial^{a}f$ is  bounded and uniformly continuous  on $\set$, for all $a\in\ind_{m}$ and $m\leq k$ , where $\ind_{m}$  is the set of all multi-indices of order $m\leq k$. This space is equipped with the following norm $||f||_{\hol^{k}(\overline{\set})}=\sum_{m=0}^{k}\sum_{a\in\ind_{m}}\sup_{x\in\set}\{|\partial^{a}f(x)|\}$, where $ \sum_{a\in\ind_{m}}$ denotes summation over all possible $m$-fold derivatives of $f$.  The operator $[\,\cdot\,]_{\hol^{0,\alpha}(\set)}$ is given by $[f]_{\hol^{0,\alpha}(\set)}\eqdef\sup_{x,y\in \set,\, x\neq y}\Bigl\{\frac{|f(x)-f(y)|}{|x-y|^{\alpha}}\Bigr\}$. We define $\hol^{k,\alpha}_{\loc}(\set)$ as the set of functions in $\hol^{k}(\set)$ such that $[\partial^{a}f]_{\hol^{0,\alpha}(\mathcal{K})}<\infty$, for all compact set $\mathcal{K}\subset\set$, $a\in\ind_{m}$ and $m\leq k$. The set $\hol^{k,\alpha}(\overline{\set})$ denotes the set of all functions in $\hol^{k}(\overline{\set})$ such that $[\partial^{a}f]_{\hol^{0,\alpha}(\set)}<\infty$, for every $a\in\mathcal{D}_{m}$ and $m\leq k$.  This set is equipped with the following norm
$||f||_{\hol^{k,\alpha}(\overline{\set})}=\sum_{m=0}^{k}\sum_{a\in\ind_{m}}[||\partial^{a}f(x)||_{\hol(\overline{\set})}+[\partial^{a}f]_{\hol^{0,\alpha}(\set)}]$.
We understand $\hol^{k,\alpha}(\R^{d})$ as $\hol^{k,\alpha}(\overline{\R^{d}})$, in the sense that  $[\partial^{a}f]_{\hol^{0,\alpha}(\R^{d})}<\infty$, for every $a\in\ind_{m}$ and $m\leq k$. As usual, $\Lp^{p}(\set)$ with $1\leq p<\infty$, denotes the class of real-valued functions on $\set$ with finite norm $||f||^{p}_{\Lp^{p}(\set)}\eqdef\int_{\set}|f|^{p}\der x<\infty$, where $\der x$ denotes the Lebesgue measure. Also, let  $\Lp^{p}_{\loc}(\set)$ consist of functions whose $\Lp^{p}$-norm is finite on any compact subset of $\set$. Define the Sobolev space $\sob^{k,p}(\set)$ as the class of functions  $f\in\Lp^{p}(\set)$ with weak or distributional partial derivatives $\partial^{a}f$, see \cite[p. 22]{adams}, and with finite norm $||f||^{p}_{\sob^{k,p}(\set)}=\sum_{m=0}^{k}\sum_{a\in\ind_{m}}||\partial^{a}f||^{p}_{\Lp^{p}(\set)}$. The space $\sob^{k,p}_{\loc}(\set)$ consists of functions whose $\sob^{k,p}$-norm is finite on any compact subset of $\set$. When $p=\infty$, the Sobolev and Lipschitz spaces are related. In particular, $\sob^{k,\infty}_{\loc}(\set)=\hol^{k-1,1}_{\loc}(\set)$  and $\sob^{k,\infty}(\set)=\hol^{k-1,1}(\overline{\set})$. Finally, $C=C(*,\dots,*)$ and $K=K(*,\dots,*)$ represent positive constants that depend only on the quantities appearing in parenthesis. 

\section{Existence and uniqueness of the NPIDD problem }\label{apr1}
 
In this section, we are interested in establishing the existence, uniqueness and regularity of the solution to the NPIDD problem  \eqref{p13.0}. The arguments used here  are based  on the Schaefer fixed point Theorem; see \cite[Thm. 4, p. 539]{evans2}.  {First}, we shall analyse some properties of $\inted w$, defined in \eqref{p6.1}, when the function $w$ is $\hol^{0,1}$ and $\hol^{1,1}$  on
$\overline{\set}_{\inted}$. These results will be helpful to show some  properties of the solution to the NPIDD problem  \eqref{p13.0}. 
\begin{rem} \label{s0}
	Notice that by definition of $\overline{\set}_{\inted}$; see \eqref{p3}, and since $s$ is a non-negative function on $\overline{\set}\times\R^{d}$, it follows that $s(x,z)=0$, for each $(x,z)\in\overline{\set}\times\R^{d}_{*}$ such that $x+z\notin\overline{\set}_{\inted}$. Then, $\inted w$ can be rewritten as
		\begin{align*}
		\inted w(x)&=\int_{\R^{d}_{*}}[w(x+z)-w(x)]s(x,z)\uno_{\{x+z\in\overline{\set}_{\inted}\}}\nu(\der z),\  \text{for}\ x\in\overline{\set}.
		\end{align*}
\end{rem}

\begin{lema}\label{g1}
\begin{enumerate}[(i)]
\item If $w\in\hol^{0,1}(\overline{\set}_{\inted})$, then $\inted w\in\hol(\overline{\set})$.

\item If $w,v\in\hol^{0,1}(\overline{\set}_{\inted})$, then $[w,v]_{\inted}=\inted [wv]-w\inted v-v\inted w$ on $\overline{\set}$, where
$$[w,v]_{\inte}\eqdef\int_{\R^{d}_{*}}[w(\cdot+z)-w][v(\cdot+z)-v]s(\cdot,z)\nu(\der z),\  \text{on}\ x\in\overline{\set}.$$

\item If $w\in\hol^{1,1}(\overline{\set}_{\inted})$, then $\inted w\in\hol^{1}(\overline{\set})$ and $\partial_{i}[\inted w]=\inted[\partial_{i}w]+\widetilde{\inte}_{i}w$
where $$\widetilde{\inte}_{i}w\eqdef\int_{\R^{d}_{*}}[w(\cdot+z)-w]\partial_{x_{i}}s(\cdot,z)\nu(\der z).$$

\end{enumerate}
\end{lema}

\begin{proof} 
Using (A4)  and by Dominated Convergence Theorem, it is easy to see that $\inted w\in\hol(\overline{\set})$, when $w\in\hol^{0,1}(\overline{\set}_{\inted})$. Calculating $[w(\cdot+z)-w][v(\cdot+z)-v]$, the reader can verify that the statements in (ii) of the  lemma above is true. We shall prove the statement given in (iii). Let  $w$ be in $\hol^{1,1}(\overline{\set}_{\inted})$, and define $f_{w}:\overline{\set}\times\R^{d}_{*}\longrightarrow\R$ as
\begin{equation*}
f_{w}(x,z)\eqdef
\begin{cases} 
[w(x+z)-w(x)]s(x,z), &\text{if}\ x+z\in\overline{\set}_{\inted},\\
0, &\text{otherwise}.
\end{cases}
\end{equation*}
 Consider $x,y\in\overline{\set}$ such that $x\neq y$. Then, by Remark \ref{s0}, we see that
\begin{align*}
|\inted w(x)-\inted w(y)|&\leq\int_{\{|z|\in(0,1)\}}|f_{w}(x,z)\uno_{\{x+z\in\overline{\set}_{\inted}\}}-f_{w}(y,z)\uno_{\{y+z\in\overline{\set}_{\inted}\}}|\nu(\der z)\\
&\quad+\int_{\{|z|\geq1\}}|f_{w}(x,z)\uno_{\{x+z\in\overline{\set}_{\inted}\}}-f_{w}(y,z)\uno_{\{y+z\in\overline{\set}_{\inted}\}}|\nu(\der z).
\end{align*}
 {Meanwhile},   from Mean Value Theorem and noting that 
	$$s(x,z)\uno_{\{y+z\in\overline{\set}_{\inted},\,x+z\notin\overline{\set}_{\inted}\}}=s(y,z)\uno_{\{x+z\in\overline{\set}_{\inted},\,y+z\notin\overline{\set}_{\inted}\}}=0,$$
we have   
\begin{align}\label{lips2}
|f_{w}(x,z)-f_{w}(y,z)|
&\leq|x-y|\bigg[\uno_{\{x+z\in\overline{\set}_{\inted},\,y+z\in\overline{\set}_{\inted}\}}\int_{0}^{1}|\deri^{1}_{x}f_{w}(y+t[x-y],z)|\der t\notag\\
&\quad+\int_{0}^{1}|\deri^{1}_{x}s(y+t[x-y],z)|\der t\Big[|w(x+z)-w(x)|\uno_{\{x+z\in\overline{\set}_{\inted},\,y+z\notin\overline{\set}_{\inted}\}}\notag\\
&\quad+|w(y+z)-w(y)|\uno_{\{y+z\in\overline{\set}_{\inted},\, x+z\notin\overline{\set}_{\inted}\}}\Big]\bigg],
\end{align}
where $\deri^{1}_{x}f_{w}$  denotes the gradient  with respect to $x$. Observe that
\begin{equation}\label{lips3}
\deri^{1}_{x}f(x,z)=s(x,z)\deri^{1}[w(x+z)-w(x)]+[w(x+z)-w(x)]\deri^{1}_{x}s(x,z),
\end{equation}
for $(x,z)\in\overline{\set}\times\R^{d}$ such that $x+z\in\overline{\set}_{\inted}$. If $|z|<1$, by \eqref{lips2}--\eqref{lips3} and using $w,\  \deri^{1}w$ are Lipschitz functions on $\overline{\set}_{\inted}$, we get 
\begin{equation}\label{lips4}
|f_{w}(x,z)-f_{w}(y,z)|
\leq K_{1} |z|\,|x-y|,
\end{equation}
where $K_{1}\eqdef\Big[\sum_{k}[\partial_{k}w]^{2}_{\hol^{0,1}(\overline{\set}_{\inted})}\Big]^{\frac{1}{2}}+3K_{2}[w]_{\hol^{0,1}(\overline{\set}_{\inted})}$ and $K_{2}\eqdef\Big[\sum_{k}||\partial_{x_{k}}s||^{2}_{\hol(\overline{\set}\times\R^{d})}\Big]^{\frac{1}{2}}$. If $|z|\geq1$, by \eqref{lips2}--\eqref{lips3} and since $w,\ \deri^{1}w$ are bounded on $\overline{\set}_{\inted}$,  it can be verified that 
\begin{align}\label{lips5}
|f_{w}(x,z)-f_{w}(y,z)|\leq K_{3}|x-y|,
\end{align}
where $K_{3}\eqdef2\Big[\sum_{k}||\partial_{k}w||^{2}_{\hol(\overline{\set}_{\inted})}\Big]^{\frac{1}{2}}+6K_{1}||w||_{\hol(\overline{\set}_{\inted})}$. Using  \eqref{lips4}--\eqref{lips5}, we have that for $(x,z)\in\overline{\set}\times\R^{d}$ such that $x+z\in\overline{\set}_{\inted}$, $\frac{1}{\varrho}|f_{w}(x+\varrho e_{i},z)-f_{w}(x,z)|$, is bounded by $K_{2}|z|\uno_{\{|z|\in(0,1)\}}+K_{3}\uno_{\{|z|\geq1\}}$, which is an integrable function with respect to the L\'evy measure $\nu$. Then, using   Dominated Convergence Theorem and \eqref{lips3}, it follows that $\partial_{i}[\inted w(x)]=\inted[\partial_{i}w(x)]+\widetilde{\inte}_{i}w(x)$. From here, (A4) and since $\partial_{i}w\in\hol^{0,1}(\overline{\set}_{\inted})$, we conclude that $\inted w\in\hol^{1}(\overline{\set})$.
\end{proof}

\subsection{ \`A priori estimates of the solution to the NPIDD problem}

To apply the Schaefer fixed point Theorem   in our problem, we need to show an \textit{\`a  priori} estimate  of   the classical solution $u^{\varepsilon}$  to  the NPIDD problem  \eqref{p13.0} on the space $(\hol^{1,\alpha'}(\overline{\set}),||\cdot||_{\hol^{1,\alpha'}(\overline{\set})})$; see Lemma \ref{gr2}. 
\begin{rem}\label{dep.1}
Notice that if the solution $u^{\varepsilon}$ is at least $\hol^{2}$ on $\overline{\set}_{\inted}$,  {then} using \cite[Thm. 3.1.22]{garroni} and the Sobolev embedding Theorem \cite[Thm. 4.12, p. 85]{adams}, it can be verified that for each $\varepsilon\in(0,1)$ fixed,
\begin{equation}\label{es.1}
||u^{\varepsilon}||_{\hol^{1,\alpha'}(\overline{\set})}\leq C\Big[||h||_{\Lp^{p'}(\set)}+||\psi_{\varepsilon}(|\deri^{1}u^{\varepsilon}|^2-g^{2})||_{\Lp^{p'}(\set)}\Big],
\end{equation}  
for some $C=C(\Lambda,\nu,s,\alpha')$, where $p'\in(d,\infty)$ is such that $\alpha'=1-\frac{d}{p'}$. We see that the second term in the RHS of \eqref{es.1} depends on $\psi_{\varepsilon}(|\deri^{1}u^{\varepsilon}|^{2}-g^{2})$. If $|\deri^{1}u^{\varepsilon}|\leq C'$ for some constant $C'$ independent of $\varepsilon$; see Lemma \ref{gr1}, then, from (A.1), \eqref{p12.1} and \eqref{es.1}, it follows that $||u^{\varepsilon}||_{\hol^{1,\alpha'}(\overline{\set})}\leq [\Lambda+\frac{1}{\varepsilon}[[C']^{2}+\Lambda^{2}+1]]C\int_{\set}\der x$. Although this estimation depends on $1/\varepsilon$, it is sufficient to use the Schaefer fixed  point Theorem in our problem, since  $\varepsilon$ is fixed; see Subsection \ref{pp1}. Later on, in Section \ref{pp2}, we will give a local estimation for $\psi_{\varepsilon}(|\deri^{1}u^{\varepsilon}|^2-g^{2})$ which is independent of $\varepsilon$; see Lemma \ref{cotaphi}. 
\end{rem}

Before continuing, we need to introduce the concepts of sub-solution and super-solution for the NPIDD problem \eqref{p13.0}.

\begin{defin}\label{princmax1}
\begin{enumerate}[1.]
  \item A function $f$ in $\hol^{2}(\overline{\set})\cap\hol^{0,1}(\overline{\set}_{\inted})$ is a sub-solution of \eqref{p13.0} if
  \begin{equation*}
  \Gamma f+ \psi_{\varepsilon}(|\deri^{1} f|^{2}- g^{2}) \leq h,\ \text{in}\ \set,\quad \text{s.t.}\ f=0,\ \text{on}\ \overline{\set}_{\inted}\setminus\set  .
    \end{equation*}
  \item  A function $f$ in  $\hol^{2}(\overline{\set})\cap\hol^{0,1}(\overline{\set}_{\inted})$ is a super-solution of \eqref{p13.0} if
 \begin{equation*}
 \Gamma f+ \psi_{\varepsilon}(|\deri^{1} f|^{2}- g^{2}) \geq h,\ \text{in}\ \set,\quad\text{s.t.}\
  f=0,\ \text{on}\  \overline{\set}_{\inted}\setminus\set. 
\end{equation*}
\end{enumerate}
 \end{defin}
An immediate consequence of  {this}  definition is the following result, which is used to prove Lemma \ref{Lb1}.
\begin{lema}\label{supsuper}
If $\varphi$ and $\eta$ are a sub-solution and a super-solution of \eqref{p13.0}, respectively, then $\varphi-\eta\leq0$  on $\overline{\set}$.
\end{lema}

\begin{proof}
From Definition \ref{princmax1}, we get
\begin{equation}\label{sup1}
\Gamma[\varphi-\eta]+ \psi_{\varepsilon}(|\deri^{1}\varphi|^{2}- g^{2})-\psi_{\varepsilon}(|\deri^{1}\eta|^{2}- g^{2})\leq 0,\ \text{in}\ \set,\quad\text{s.t.}\  \varphi-\eta=0,\ \text{on}\   \overline{\set}_{\inted}\setminus\set.
\end{equation}
Let $x^{*}\in\overline{\set}$ be a maximum point of $\varphi-\eta$. If $x^{*}\in\partial{\set}$, trivially, we have $\varphi-\eta\leq0$ on $\overline{\set}$. Suppose that $x^{*}\in \set$. This means that 
\begin{equation}\label{sup1.1}
\begin{split}
&\deri^{1}[\varphi-\eta](x^{*})=0,\ \tr[a(x^{*})\deri^{2}[\varphi-\eta](x^{*})]\leq 0,\\
&[\varphi-\eta](x^{*}+z)-[\varphi-\eta](x^{*})\leq 0 ,\ \text{for}\ z\in\R^{d}_{*}\  \text{with}\  x^{*}+z\in\overline{\set}_{\inted}. 
\end{split}
\end{equation}
Since $[\varphi-\eta](x^{*}+z)=0$ when $x^{*}+z\in\overline{\set}_{\inted}\setminus\set$, it follows that $[\varphi-\eta](x^{*})\geq0$.  {Meanwhile}, applying \eqref{sup1.1} in \eqref{sup1}, it yields $c(x^{*})[\varphi-\eta](x^{*})\leq0$. Then, $[\varphi-\eta](x^{*})\leq0$, since $c>0$ on $\overline\set$. Therefore, $[\varphi-\eta](x)\leq[\varphi-\eta](x^{*})=0$ for all $x\in\overline{\set}$.
\end{proof}

\begin{lema}\label{Lb1}
If $u^{\varepsilon}\in\hol^{2}(\overline{\set})\cap\hol^{0,1}(\overline{\set}_{\inted})$ is a solution to the NPIDD problem \eqref{p13.0}, there exists a positive constant  $C_{1}$ independent of $\varepsilon$, such that $0\leq u^{\varepsilon}\leq C_{1}$ on $\overline{\set}$ and $|\deri^{1}u^{\varepsilon}|\leq d^{\frac{1}{2}} C_{1}$ on $\partial \set$.
\end{lema}

From now on, for simplicity of notation, we replace $u^{\varepsilon}$ by $u$ in the proofs of the results. 
 
\begin{proof}[Proof of Lemma \ref{Lb1}]
For $h$ which is a $\hol^{1,\alpha'}$-function on $\overline{\set}$, let $v\in\hol^{2,\alpha'}(\overline{\set})$ be  the unique solution to the LPIDD problem
$$\Gamma v=h,\ \text{in}\ \set,\quad \text{s.t.}\  v=0,\ \text{on}\  \overline{\set}_{\inted}\setminus\set.$$
Then, $||v||_{\hol^{2,\alpha'}(\overline{\set})}\leq K_{4} ||h||_{\hol^{0,\alpha'}(\overline{\set})}\leq K_{4}\Lambda\defeq C_{1}$, where $K_{4}=K_{4}(d,\Lambda,\nu,s,\alpha')$; see \cite[Thm. 3.1.12]{garroni}. Then $v$ is a super-solution of \eqref{p13.0}.  {Meanwhile}, we know that $h\geq0$, this implies that the zero function is a sub-solution of \eqref{p13.0}. Therefore, using Lemma \ref{supsuper}, it follows that $0\leq u\leq K_{4}\Lambda$ on $\overline{\set}$. Take a point $x$ in $\partial \set$ and a unit vector $\mathbb{n}_{x}$ outside $\set$ such that it is not tangent to $\set$. Defining $\mathbb{v}=-\mathbb{n}_{x}$, we have that $\langle\mathbb{v},\deri^{1}v(x)\rangle=\lim_{\varrho\rightarrow0}\frac{v(x+\varrho\mathbb{v})}{\varrho}\geq\lim_{\varrho\rightarrow0}\frac{u(x+\varrho\mathbb{v})}{\varrho}=\langle\mathbb{v},\deri^{1}u(x)\rangle$. Since $\mathbb{v}=-\mathbb{n}_{x}$, it yields $\langle\mathbb{n}_{x},\deri^{1}v(x)\rangle\leq\langle\mathbb{n}_{x},\deri^{1}u(x)\rangle\leq 0$. Then, $|\langle\mathbb{n}_{x},\deri^{1}u(x)\rangle|\leq \Big(\sum_{i}||\partial_{i}v||^{2}_{\hol(\overline{\set})}\Big)^{\frac{1}{2}}\leq d^{\frac{1}{2}} C_{1}$. Suppose that $|\deri^{1}u(x)|\neq0$ and that the  vector $\frac{\deri^{1}u(x)}{|\deri^{1}u(x)|}$ is outside $\set$.  Taking $\mathbb{n}_{x}=\frac{\deri^{1}u(x)}{|\deri^{1}u(x)|}$, it follows that $|\deri^{1}u(x)|\leq  d^{\frac{1}{2}}C_{1}$. If the vector $\mathbb{v}=\frac{\deri^{1}u(x)}{|\deri^{1}u(x)|}$ is inside  $\set$, proceeding as before, we have $0\leq\langle\mathbb{v},\deri^{1}u(x)\rangle\leq\langle\mathbb{v},\deri^{1}v(x)\rangle$. Therefore, $|\deri^{1}u(x)|\leq  d^{\frac{1}{2}}C_{1}$. In the case that $|\deri^{1}u(x)|=0$, the inequality is trivially true.  {Consequently, we have} finished the proof.
\end{proof}

Before checking  $|\deri^{1}u^{\varepsilon}|\leq C'$ on $\overline{\set}$, for some constant $C'>0$ independent of $\varepsilon$, we need  to define an auxiliary function $\varphi$, which satisfies \eqref{partu4} on $\set$. In particular, \eqref{partu4} is true, when $\varphi$ is evaluated at its maximum $x^{*}\in\set$, which helps us to prove Lemma \ref{gr1}.  

\begin{lema}\label{lemfrontera1.0}
Let $u^{\varepsilon}\in\hol^{3}(\overline{\set})\cap\hol^{2}(\overline{\set}_{\inted})$ be a solution to the NPIDD problem \eqref{p13.0}. Define the auxiliary function $\varphi:\overline{\set}_{\inted}\longrightarrow\R$ as 
$\varphi\eqdef|\deri^{1}u^{\varepsilon}|^{2}-\lambda M_{\varepsilon}u^{\varepsilon}$, on $\overline{\set}_{\inted}$, where  $M_{\varepsilon}\eqdef\sup_{x\in\overline{\set}}|\deri^{1}u^{\varepsilon}(x)|$ and $\lambda>0$. Then, {$\varphi\in\hol^{2}(\overline{\set}_{\inted})$ and} there exists a positive constant  $C_{2}$ independent of $\varepsilon$ such that
\begin{multline}\label{partu4}
-\tr[a\deri^{2}\varphi]-\inte\varphi\leq C_{2}|\deri^{1}u^{\varepsilon}|^{2}+C_{2} [1+M_{\varepsilon}[1+\lambda]]|\deri^{1}u^{\varepsilon}|+\lambda C_{2} M_{\varepsilon}\\
-\psi'_{\varepsilon}(\cdot)[2\langle\deri^{1}u^{\varepsilon},\deri^{1}|\deri^{1}u^{\varepsilon}|^{2}\rangle-C_{2}|\deri^{1}u^{\varepsilon}|-\lambda M_{\varepsilon} [|\deri^{1}u^{\varepsilon}|^{2}- g^{2}]],\ \text{on}\ \set,
\end{multline}
where $\psi'_{\varepsilon}(\cdot)$ denotes $\psi'_{\varepsilon}(|\deri^{1}u|^{2}- g^{2})$.
\end{lema}
\begin{proof}
Notice that $\varphi\in\hol^{2}(\overline{\set}_{\inted}\setminus\partial\set)\cap\hol^{1}(\overline{\set}_{\inted})$, $\partial_{i}\varphi=2\langle\deri^{1}\partial_{i}u,\deri^{1}u\rangle-\lambda M_{\varepsilon}\partial_{i}u$ and $\partial_{ij}\varphi=2[\langle\deri^{1}\partial_{ij}u,\deri^{1}u\rangle+\langle\deri^{1}\partial_{i}u,\deri^{1}\partial_{j}u\rangle]-\lambda M_{\varepsilon}\partial_{ij}u$ on $\overline{\set}$. Then, from here and using $u=\partial_{i}u=\partial_{ij}u=0$ on $\partial\set$, it is easy to see that $\partial_{ij}\varphi=0$ on $\partial\set$ and thus $\varphi\in\hol^{2}(\overline{\set}_{\inted})$. Observe that
\begin{align}\label{dercot1}
-\tr[a\deri^{2}\varphi]-\inte\varphi&=-2\sum_{k}\langle a\deri^{1}\partial_{k}u,\deri^{1}\partial_{k}u\rangle-2\sum_{k}\tr [a\deri^{2}\partial_{k}u]\partial_{k}u\notag\\
&\quad-\inte|\deri^{1}u|^{2}+\lambda M_{\varepsilon}[\tr[a\deri^{2}u]+\inte u].
\end{align}
From \eqref{p6} and \eqref{p13.0}, it follows that
\begin{align}
\lambda M_{\varepsilon}[\tr[a\deri^{2}u]+\inte u]&=\lambda M_{\varepsilon}\psi_{\varepsilon}(\cdot)+\lambda M_{\varepsilon}[\widetilde{\deri}_{1}u-h],\label{dercot2.0}
\end{align}
where $\widetilde{\deri}_{1}u\eqdef\langle b,\deri^{1}u\rangle+cu$. {Differentiating \eqref{p13.0} and by Lemma \ref{g1}.iii, we see that} 
\begin{align}
-\tr[a\deri^{2}\partial_{k}u]-\inte\partial_{k}u&=\tr[[\partial_{k}a]\deri^{2}u]+\widetilde{\inte}_{k}u\notag\\
&\quad+\partial_{k}[h-\langle b,\deri^{1}u\rangle-cu]-\psi'_{\varepsilon}(\cdot)\partial_{k}[|\deri^{1}u|^{2}-g^{2}].\label{dercot2.0.0}
\end{align}  
 By Lemma \ref{g1}.ii, it yields
\begin{equation}\label{der.1}	
-2\partial_{k}u\inte\partial_{k}u=[\partial_{k}u,\partial_{k}u]_{\inte}-\inte[\partial_{k}u]^{2}. 
\end{equation}
From here,  {multiplying  \eqref{dercot2.0.0} by $2\partial_{k}u$ and taking summation over all $k$'s,}
\begin{multline}\label{dercot2}
-2\sum_{k}\tr[a\deri^{2}\partial_{k}u]\partial_{k}u-\inte|\deri^{1}u|^{2}\\
=\widetilde{\deri}_{2}u +2\langle\deri^{1}u,\deri^{1}h\rangle-2\psi'_{\varepsilon}(\cdot)\langle\deri^{1}u,\deri^{1}[|\deri^{1}u|^{2}-g^{2}]\rangle-\sum_{k}[\partial_{k}u,\partial_{k}u]_{\inte},
\end{multline}
with   $\widetilde{\deri}_{2}u\eqdef 2[\sum_{k}\partial_{k}u[\tr[[\partial_{k}a]\deri^{2}u]+\widetilde{\inte}_{k}u]
-\langle\deri^{1}u,\deri^{1}[\langle b,\deri^{1}u\rangle+cu]\rangle]$. Applying \eqref{dercot2.0} and \eqref{dercot2} in \eqref{dercot1}, we get 
\begin{align}\label{in1}
-\tr[a\deri^{2}\varphi]-\inte\varphi&=-2\psi'_{\varepsilon}(\cdot)\langle\deri^{1}u,\deri^{1}[|\deri^{1}u|^{2}-g^{2}]\rangle+\lambda M_{\varepsilon}\psi_{\varepsilon}(\cdot)\notag\\
&\quad-2\sum_{k}\langle a\deri^{1}\partial_{k}u,\deri^{1}\partial_{k}u\rangle-\lambda M_{\varepsilon}h+2\langle\deri^{1}u,\deri^{1}h\rangle\notag\\
&\quad-\sum_{k}[\partial_{k}u,\partial_{k}u]_{\inte}+\widetilde{\deri}_{2}u +\lambda M_{\varepsilon}\widetilde{\deri}_{1}u.
\end{align}
Notice that
\begin{equation}\label{in.1}
\lambda M_{\varepsilon}\psi_{\varepsilon}(\cdot)\leq \lambda M_{\varepsilon}\psi'_{\varepsilon}(\cdot)[|\deri^{1}u|^{2}- g^{2}],
\end{equation}
since $\psi_{\varepsilon}$ is a convex function. From (A1), (A3) and since $h\geq0$ on $\overline{\set}$, it follows  {that}
\begin{equation}
-2\sum_{k}\langle a\deri^{1}\partial_{k}u,\deri^{1}\partial_{k}u\rangle-\lambda M_{\varepsilon} h+2\langle\deri^{1}h,\deri^{1}u\rangle\leq -2\theta|\deri^{2}u|^{2}+2d\Lambda|\deri^{1}u|.
\end{equation}
By (A4), Mean Value Theorem and the estimate  $0\leq u\leq C_{1}$, with $C_{1}$ as in Lemma \ref{Lb1}, we get 
\begin{align*}
{2\sum_{k}\partial_{k}u}\widetilde{\inte}_{k}u&\leq 2dK_{2}{|\deri^{1}u|}\Bigg[\int_{\{|z|\in(0,1)\}}\bigg[\int_{0}^{1}|\deri^{1}u(\cdot+tz)|\der t\bigg]|z|\nu(\der z)\notag\\
&\quad+\int_{\{|z|\geq1\}}[u(\cdot+z)+u]\nu(\der z)\Bigg]\leq 2dK_{2}C_{\nu}[M_{\varepsilon}+2C_{1}]{|\deri^{1}u|}, 
\end{align*}
where $C_{\nu}, K_{2}$ are as in \eqref{H4}, \eqref{lips4}, respectively. 
{Then, using the inequality above} and since $[\partial_{k}u,\partial_{k}u]_{\inte}\geq0$, we have 
\begin{multline}\label{in2}
-\sum_{k}[\partial_{k}u,\partial_{k}u]_{\inte}+\widetilde{\deri}_{2}u +\lambda M_{\varepsilon} \widetilde{\deri}_{1}u
\leq 4d^{3}\Lambda|\deri^{2}u|\,|\deri^{1}u|+2d^{2}\Lambda|\deri^{1}u|^{2}\\
+2dC_{1}[\Lambda+2K_{2}C_{\nu}]|\deri^{1}u|+dM_{\varepsilon}[\lambda\Lambda+2K_{2}C_{\nu}]|\deri^{1}u|+\lambda \Lambda M_{\varepsilon}C_{1}. 
\end{multline}
Therefore, applying  \eqref{in.1}--\eqref{in2} in \eqref{in1} and  noting that
$-\theta|\deri^{2}u|^{2}+2d^{3}\Lambda|\deri^{2}u|\,|\deri^{1}u|\leq \frac{[d^{3}\Lambda]^{2}}{{\theta}}|\deri^{1}u|^{2}$, we obtain the inequality \eqref{partu4}, where $C_{2}=C_{2}(d,\Lambda,\nu,s,\alpha')$. 
\end{proof}

\begin{lema}\label{gr1}
If $u^{\varepsilon}\in\hol^{3}(\overline{\set})\cap\hol^{2}(\overline{\set}_{\inted})$ is a solution to the NPIDD problem \eqref{p13.0},  {then} there exists a positive constant $C_{3}$ independent of $\varepsilon$, such that $|\deri^{1}u^{\varepsilon}|\leq C_{3}$ on $\overline{\set}$.
\end{lema}

\begin{proof}
Let  $\varphi$ and $M_{\varepsilon}$ be as in Lemma \ref{lemfrontera1.0}, with  $\lambda\geq1$ a constant that shall be selected later on. Observe that if $M_{\varepsilon}\leq1$, we obtain a bound for $M_{\varepsilon}$ that is independent of $\varepsilon$.  We assume henceforth that $M_{\varepsilon}>1$. Let $x^{*}\in\overline{\set}$ be a point where $\varphi$ attains its maximum on $\set$. Then,
\begin{align}\label{partu5}
|\deri^{1}u(x)|^{2}\leq |\deri^{1}u(x^{*})|^{2}+\lambda M_{\varepsilon}[u(x)-u(x^{*})]\leq|\deri^{1}u(x^{*})|^{2}+\lambda M_{\varepsilon}C_{1},
\end{align}
for all $x\in \overline{\set}$. The last inequality in  \eqref{partu5}  is obtained from Lemma \ref{Lb1}.  If $x^{*}\in\partial \set$,  by  Lemma \ref{Lb1}, it is easy to  deduce  $\varphi(x^{*})=|\deri^{1}u(x^{*})|^{2}\leq dC_{1}^{2}$. Then, from  \eqref{partu5}, $|\deri^{1}u|^{2}\leq dC_{1}^{2}+\lambda M_{\varepsilon} C_{1}$ in $\overline{\set}$. Notice that for all $\epsilon$, there exists $x_{0}\in\overline{\set}$ such that $[M_{\varepsilon}-\epsilon]^{2}\leq|\deri^{1} u(x_{0})|^{2}$. Then,
\begin{align}\label{partu5.0}
[M_{\varepsilon}-\epsilon]^{2}\leq dC^{2}_{1}+\lambda M_{\varepsilon} C_{1}.
\end{align}
Letting $\epsilon\rightarrow0$ in  \eqref{partu5.0}, it follows that
$|\deri^{1}u|\leq M_{\varepsilon}\leq \frac{dC^{2}_{1}}{M_{\varepsilon}}+\lambda C_{1}\leq dC^{2}_{1}+\lambda C_{1}$, since $M_{\varepsilon}>1$. Let $x^{*}$ be in $\set$. We have $\tr[a(x^{*})\deri^{2}\varphi(x^{*})]\leq0$, $\varphi(x^{*})\geq\varphi(x^{*}+z)$ for $x^{*}+z\in\overline{\set}_{\inted}$ and $\partial_{i}\varphi(x^{*})=\partial_{i}|\deri^{1}u(x^{*})|^{2}-\lambda M_{\varepsilon}\partial_{i}u(x^{*})=0$. Then, $0\leq-\tr[a\deri^{2}\varphi]-\inte\varphi$ and $ 2\langle\deri^{1}u,\deri^{1}|\deri^{1}u|^{2}\rangle=2\lambda M_{\varepsilon}|\deri^{1}u|^{2}$ at $x^{*}$. From here, using Lemma \ref{lemfrontera1.0} and since $\psi'_{\varepsilon}(\cdot)\geq0$, it follows that
\begin{align}\label{partu6}
0&\leq C_{2}|\deri^{1}u|^{2}+C_{2} [1+M_{\varepsilon}[1+\lambda]]|\deri^{1}u|\notag\\
&\quad+\lambda C_{2} M_{\varepsilon}-\psi'_{\varepsilon}(\cdot)[\lambda M_{\varepsilon}|\deri^{1}u|^{2}-C_{2}|\deri^{1}u|+\lambda M_{\varepsilon} g^{2}],\ \text{at}\ x^{*}, 
\end{align}
where $C_{2}$ is as in Lemma \ref{lemfrontera1.0}. If $\psi'_{\varepsilon}(\cdot)<1<\frac{1}{\varepsilon}$, by definition of $\psi_{\varepsilon}$, given in  \eqref{p12.1}, we obtain that $\psi_{\varepsilon}(\cdot)\leq1$. 
It follows that 
$|\deri^{1}u(x^{*})|^{2}\leq 2+\Lambda^{2}$.
Then, by  \eqref{partu5}  and arguing as  in  \eqref{partu5.0}, we obtain $M_{\varepsilon}\leq 2+\Lambda^{2}+\lambda C_{1}$. If $\psi'_{\varepsilon}(\cdot)\geq1$, then, multiplying by $\frac{1}{M_{\varepsilon}\psi_{\varepsilon}'(\cdot)}$ in \eqref{partu6}, it can be verified that 
\begin{equation} \label{partu7}
0\leq [C_{2}-\lambda]|\deri^{1}u|^{2}+C_{2}[3+\lambda]|\deri^{1}u|+\lambda C_{2},\  \text{at}\ x^{*}.
\end{equation}
 {Notice that  {this} inequality is satisfied for any $
\lambda\geq1$ fixed, where the maximum point of $\varphi$, $x^{*}\in\set$, depends on $\lambda$. Then, taking $\lambda\geq\max\{1,C_{2}\}$ fixed, from \eqref{partu7}, it follows that $|\deri^{1}u(x^{*})|<K_{5}$, for some $K_{5}=K_{5}(d,\Lambda,\nu,s,\alpha',\lambda)$.} Using  \eqref{partu5} and an argument similar to \eqref{partu5.0}, we conclude that there exists $C_{3}=C_{3}(d,\Lambda,\nu,s,\alpha',\lambda)$ such that   $|\deri^{1}u|\leq M_{\varepsilon}\leq C_{3}$ on $\overline{\set}$.
\end{proof}

By the previous results seen here and using  \eqref{es.1}, we obtain the following estimate  of $u^{\varepsilon}$ on the space $(\hol^{1,\alpha'}(\overline{\set}),||\cdot||_{\hol^{1,\alpha'}(\overline{\set})})$.

\begin{lema}\label{gr2}
If $u^{\varepsilon}\in\hol^{3}(\overline{\set})\cap\hol^{2}(\overline{\set}_{\inted})$ is a solution to the NPIDD problem \eqref{p13.0}, there exists $C_{4}=C_{4}(\varepsilon,\Lambda,\nu,s,\alpha')$ such that
$||u^{\varepsilon}||_{\hol^{1,\alpha'}(\overline{\set})}\leq C_{4}$.
\end{lema}
\begin{proof}
By Remark \ref{dep.1} and Lemma \ref{gr1}, it yields $||u||_{\hol^{1,\alpha'}(\overline{\set})}\leq C\int_{\set}\der x[\Lambda+K_{6}[C_{3}^{2}+\Lambda^{2}+1]]$ for some $C=C(\Lambda,\nu,s,\alpha')$ and  $K_{6}=K_{6}(\varepsilon,\Lambda)$. Taking $C_{4}\eqdef C\int_{\set}\der x[\Lambda+K_{6}[C_{3}^{2}+\Lambda^{2}+1]]$, it follows that $||u||_{\hol^{1,\alpha'}(\overline{\set})}\leq C_{4}$ on $\overline{\set}$.
\end{proof}

\subsection{Proof of Proposition \ref{princ1.0}}\label{pp1}

In this  subsection, we present the proof  of the existence, uniqueness and regularity of the solution $u^{\varepsilon}$ to the NPIDD problem \eqref{p13.0}. 

\begin{proof}[Proof of Proposition \ref{princ1.0}]
Let $\varepsilon\in(0,1)$ be fixed. Observe that the following LPIDD problem 
\begin{equation}\label{p13.0.0}
\Gamma  \mathpzc{u}=h- \psi_{\varepsilon}(|\deri^{1} w|^{2}- g^{2}),\ \text{in}\ \set,\quad\text{s.t.}\ 
   \mathpzc{u}=0,\  \text{on}\ \overline{\set}_{\inted}\setminus\set  ,
\end{equation}
has a unique solution $\mathpzc{u}\in\hol^{2,\alpha'}(\overline{\set})$, for each $w\in\hol^{1,\alpha'}(\overline{\set})$, since $h- \psi_{\varepsilon}(|\deri^{1} w|^{2}- g^{2})\in\hol^{0,\alpha'}(\overline{\set})$, and (A1), (A3) and (A4) hold; see \cite[Thm. 3.1.12]{garroni}. Defining the map $T:\hol^{1,\alpha'}(\overline{\set})\longrightarrow\hol^{2,\alpha'}(\overline{\set})$ as 
$T[w]=\mathpzc{u}$, for each $w\in\hol^{1,\alpha'}(\overline{\set})$, where $\mathpzc{u}$ is the solution to the LPIDD problem \eqref{p13.0.0}, we see that $T$ is well defined. Notice that $T$ is continuous and maps bounded sets in $\hol^{1,\alpha'}(\overline{\set})$ into bounded sets in $\hol^{2,\alpha'}(\overline{\set})$ which are pre-compact in the H\"older space  $(\hol^{1,\alpha'}(\overline{\set}),||\cdot||_{\hol^{1,\alpha'}(\overline{\set})})$; see \cite[Thm. 16.2.2]{Cs1}. To use the Schaefer fixed point Theorem; see \cite[Thm. 4, p. 539]{evans2}, we need to verify that the set 
$\widetilde{\mathcal{A}}\eqdef\{w\in\hol^{1,\alpha'}(\overline{\set}): \rho T[w]=w,\ \text{for some}\ \rho\in[0,1]\}$, is bounded uniformly, i.e. $||w||_{\hol^{1,\alpha'}(\overline{\set})}\leq C$ for all $w\in\widetilde{\mathcal{A}}$, where $C$ is some positive constant which is independent of $w$ and $\rho$. Let $w$ be in $\widetilde{\mathcal{A}}$. Notice that if $\rho=0$,  {then} it follows immediately that $w=0$. So, assume $w\in\hol^{1,\alpha'}(\overline{\set})$ such that $ T[w]=\frac{w}{\rho}$, for some $\rho\in(0,1]$; or, in other words, $w\in\hol^{2,\alpha'}(\overline{\set})$, since the map $T$ is defined from $\hol^{1,\alpha'}(\overline{\set})$ to $\hol^{2,\alpha'}(\overline{\set})$, and  
\begin{equation}\label{p13.0.1}
\Gamma  w=\rho[h-\psi_{\varepsilon}(|\deri^{1}w|^{2}- g^{2})],\ \text{in}\ \set,\quad\text{s.t.}\ 
   w=0,\ \text{on}\ \overline{\set}_{\inted}\setminus\set.
\end{equation}
Taking
 $f\eqdef \rho[ h-\psi_{\varepsilon}(|\deri^{1}w|^{2}- g^{2})]+\inte w$,  from (A1), (A4), Lemma \ref{g1} and since $\rho[h-\psi_{\varepsilon}(|\deri^{1}w|^{2}- g^{2})]\in\hol^{1,\alpha'}(\overline{\set})$, we have that $f\in\hol^{1,\alpha'}(\overline{\set})$. Then, the linear Dirichlet problem 
\begin{equation}\label{p13.0.3}
\dif\tilde{v}=f,\ \text{in}\ \set,\quad\text{s.t.}\ 
 \tilde{v}=0, \ \text{on}\ \partial \set,
\end{equation}
has a unique solution $\tilde{v}\in\hol^{3,\alpha'}(\overline{\set})$, since (A3) holds  {and the boundary $\partial\set$ is of class $\hol^{3,\alpha'}$}; see \cite[Thms. 6.14, 9.19, pp. 107, 244, respectively]{gilb}. Recall that the elliptic differential operator $\dif$ is defined in \eqref{p6.1}. From \eqref{p6} and \eqref{p13.0.1}--\eqref{p13.0.3}, it follows that 
\begin{equation*}
  \dif w=\dif \tilde{v},\ \text{in}\ \set,\quad\text{s.t.}\ 
  w=\tilde{v},\ \text{on}\ \partial\set.
\end{equation*}
 {From here and using \cite[Thm. 6.14 p. 107]{gilb}, $w=\tilde{v}$ in $\set$, and hence $w\in\hol^{3,\alpha'}(\overline{\set})$}. Therefore $\widetilde{\mathcal{A}}\subset\hol^{3,\alpha'}(\overline{\set})$. Now,  applying  similar arguments, seen in proofs of Lemmas \ref{Lb1}, \ref{gr1} and \ref{gr2},  to \eqref{p13.0.1}, it can be verified that $0\leq w\leq C_{1}$, $|\deri^{1}w|\leq C_{3}$, on $\overline{\set}$, and $||w||_{\hol^{1,\alpha'}(\overline{\set})}\leq C_{4}$, where $C_{1}, C_{3}, C_{4}$ are positive constants as in  Lemmas \ref{Lb1}, \ref{gr1} and \ref{gr2}, respectively. Notice that these constants are independent of $\rho$ and $w$. This means that $\widetilde{\mathcal{A}}$ is bounded uniformly on $(\hol^{1,\alpha'}(\overline{\set}),||\cdot||_{\hol^{1,\alpha'}(\overline{\set})})$. Since $T$ is a continuous and compact mapping from the Banach space $(\hol^{1,\alpha'}(\overline{\set}),||\cdot||_{\hol^{1,\alpha'}(\overline{\set})})$ to itself and the set $\widetilde{\mathcal{A}}$ is bounded uniformly, we conclude, by the Schaefer fixed point Theorem, there exists a fixed point $u^{\varepsilon}\in\hol^{1,\alpha'}(\overline{\set})$ to the problem $T[u^{\varepsilon}]=u^{\varepsilon}$ which satisfies the NPIDD problem \eqref{p13.0}.   {In addition}, we have $u^{\varepsilon}=T[u^{\varepsilon}]\in\hol^{2,\alpha'}(\overline{\set})$ and by  similar arguments seen previously, it can be shown that  $u^{\varepsilon}$ is non-negative and belongs to $\hol^{3,\alpha'}(\overline{\set})$. The uniqueness of the solution $u^{\varepsilon}$ to the problem \eqref{p13.0}, is obtained from Lemma \ref{supsuper}. With this remark we finish the proof.  
\end{proof}

\section{Existence and uniqueness of the HJB equation}\label{pp2}

Since $u^{\varepsilon}$  satisfies Lemmas \ref{Lb1} and \ref{gr1}, for each $\varepsilon\in(0,1)$, and the constants that appear in these Lemmas are independent of $\varepsilon$, we only need  to show that $\psi_{\varepsilon}(|\deri^{1}u^{\varepsilon}|^{2}- g^{2})$ is locally bounded by a positive constant independent of $\varepsilon$; see Lemma \ref{cotaphi}. This estimate implies  that $u^{\varepsilon}$ is locally bounded uniformly in $\varepsilon$, i.e., $||u^{\varepsilon}||_{\sob^{2,p}(B_{\beta r})}\leq C$ where $B_{\beta r}\subset\set$ is an open ball with radius $\beta r$, where $\beta\in(0,1]$ and $r>0$;  see Lemmas \ref{lemafrontera3} and \ref{lemafrontera4}. From here, we extract a convergent sub-sequence $\{u^{\varepsilon_{\kappa}}\}_{\kappa\geq1}$ of $\{u^{\varepsilon}\}_{\varepsilon\in(0,1)}$, whose limit function is the solution to the HJB equation \eqref{p1}; see Subsection \ref{proofHJB1}.   

\subsection{Some local properties of the solution to the NPIDD problem}

Before showing that $\psi_{\varepsilon}(|\deri^{1}u^{\varepsilon}|^{2}- g^{2})\leq C$ is   locally bounded by  some  constant $C$ independent of $\varepsilon$, we need  to define an auxiliary function $\phi$, which   satisfies \eqref{gr3.2} on $B_{\beta' r}$, with $\beta'$ as in Remark \ref{R1}. In particular, \eqref{gr3.2} is true, when $\phi$ is evaluated at its maximum $x^{*}\in B_{\beta' r}\subset\set$, which helps us to prove Lemma \ref{cotaphi}.  

\begin{rem}\label{R1}
Let $\overline{\set}^{\prime}_{\inted}$ be the interior set of $\overline{\set}_{\inted}$. In Lemmas \ref{gr3}--\ref{lemafrontera4} and their proofs, we  consider cut-off functions $\xi\in\hol^{\infty}_{\comp}(\overline{\set}^{\prime}_{\inted})$  which satisfy $0\leq\xi\leq1$, $\xi=1$ on the open ball $B_{\beta r}\subset B_{\beta'r}\subset\set$ and $\xi=0$ on $\overline{\set}^{\prime}_{\inted} \setminus B_{\beta' r}$, with $r>0$, $\beta'= \frac{\beta+1}{2}$ and $\beta\in(0,1]$.  It is  {also} assumed that $||\xi||_{\hol^{2}(\overline{B_{\beta r}})}\leq K_{7}$, where $K_{7}>0$ is a constant independent of $\varepsilon$.
\end{rem}

\begin{lema}\label{gr3}
If 
$\phi\eqdef\xi\psi_{\varepsilon}(|\deri^{1}u^{\varepsilon}|^{2}- g^{2})$ on $\overline{\set}^{\prime}_{\inted}$, then
\begin{align}\label{gr3.2}
\tr[a\deri^{2}\phi]&\geq\psi'_{\varepsilon}(\cdot)[\theta\xi|\deri^{2}u^{\varepsilon}|^{2}-C_{5}|\deri^{2}u^{\varepsilon}|-C_{5}]\notag\\
&\quad-\xi\psi'_{\varepsilon}(\cdot)\inte|\deri^{1}u^{\varepsilon}|^{2}+2\psi'_{\varepsilon}(\cdot)\langle\deri^{1}\phi,\deri^{1}u^{\varepsilon}\rangle,\ \text{on}\ B_{\beta' r},
\end{align}
 for some positive constant $C_{5}$ independent of $\varepsilon$, where $\theta$ is given in (A3).
\end{lema}

\begin{proof}
Let  $\phi$ be as in  Lemma \ref{gr3}.  First and second derivatives of $\phi$ on $B_{\beta' r}\subset\set$, are given by 
\begin{equation}\label{p7}
\begin{split}
\partial_{i}\phi&=\psi_{\varepsilon}(\cdot)\partial_{i}\xi+\xi\psi'_{\varepsilon}(\cdot)[2\langle\deri^{1}u,\deri^{1}\partial_{i}u\rangle-\partial_{i}[ g^{2}]],\\
\partial_{ij}\phi&=\psi_{\varepsilon}(\cdot)\partial_{ij}\xi+\psi'_{\varepsilon}(\cdot)[\partial_{i}\xi\partial_{j}[|\deri^{1}u|^{2}- g^{2}]+\partial_{j}\xi\partial_{i}[|\deri^{1}u|^{2}- g^{2}]]\\
&\quad+\xi\psi'_{\varepsilon}(\cdot)[2[\langle\deri^{1}\partial_{i}u,\deri^{1}\partial_{j}u\rangle+\langle\deri^{1}u,\deri^{1}\partial_{ij}u\rangle]-\partial_{ij}[ g^{2}]]+4\xi\psi''_{\varepsilon}(\cdot)\bar{\eta}_{i}\bar{\eta}_{j},
\end{split}
\end{equation}
where $\bar{\eta}=(\bar{\eta}_{1},\dots,\bar{\eta}_{d})$ with $\bar{\eta}_{i}\eqdef\langle\deri^{1}u,\deri^{1}\partial_{i}u\rangle-\frac{\partial_{i}[ g^{2}]}{2}$. Then,
\begin{align}\label{p8}
\tr[a\deri^{2}\phi]&=2\xi\psi'_{\varepsilon}(\cdot)\sum_{k}\partial_{k}u\tr[a\deri^{2}\partial_{k}u]+4\xi\psi''_{\varepsilon}(\cdot)\langle a\bar{\eta},\bar{\eta}\rangle+\psi_{\varepsilon}(\cdot)\tr[a\deri^{2}\xi]\notag\\
&\quad+\xi\psi'_{\varepsilon}(\cdot)\biggl[2\sum_{ij}a_{ij}\langle\deri^{1}\partial_{i}u,\deri^{1}\partial_{j}u\rangle-\tr[a\deri^{2}[ g^{2}]]\biggr]\notag\\
&\quad+2\psi'_{\varepsilon}(\cdot)\biggr[2\sum_{ij} a_{ij}\partial_{i}\xi\langle\deri^{1}u,\deri^{1}\partial_{j}u\rangle-\langle a\deri^{1}\xi,\deri^{1}[ g^{2}]\rangle\biggl].
\end{align}
Using \eqref{dercot2.0.0}, \eqref{der.1} and \eqref{p7}, it can be verified 
\begin{align}\label{p9}
2\xi\psi'_{\varepsilon}(\cdot)\sum_{k}\partial_{k}u\tr[a\deri^{2}\partial_{k}u]&=\psi'_{\varepsilon}(\cdot)[\xi\widetilde{\deri}_{3}u+2\langle\deri^{1}\phi,\deri^{1}u\rangle-2\psi_{\varepsilon}(\cdot)\langle\deri^{1}\xi,\deri^{1}u\rangle]\notag\\
&\quad+\xi\psi'_{\varepsilon}(\cdot)\bigg[\sum_{k}[\partial_{k}u,\partial_{k}u]_{\inted}-\inted|\deri^{1}u|^{2}\bigg],
\end{align}
with $\widetilde{\deri}_{3}u\eqdef-2\sum_{k}\partial_{k}u[\tr[[\partial_{k}a]\deri^{2}u]+\widetilde{\inted}_{k}u]-2\langle\deri^{1}u,\deri^{1}[h-\langle b,\deri^{1}u\rangle-cu]\rangle$. Observe that 
\begin{equation}\label{p.9}
\xi\widetilde{\deri}_{3}u\geq -2d\Lambda C_{3}[d[d+1]|\deri^{2}u|+2dC_{3}+C_{1}+1]-2dK_{2}C_{\nu}C_{3}[C_{3}+2C_{1}],
\end{equation}
where $\Lambda,C_{\nu},  K_{2}, C_{1},C_{3},K_{7}$ are as in (A1), (A4), \eqref{lips4}, Lemmas \ref{Lb1}, \ref{gr1}, and Remark \ref{R1}, respectively. By \eqref{p6.1} and \eqref{p13.0}, it yields $\psi_{\varepsilon}(\cdot)=h+\tr[a\deri^{2}u]-\langle b,\deri^{1}u\rangle -cu+\inted u$ on $\set$. From here, 
we see that
\begin{align}\label{p9.0}
\psi_{\varepsilon}(\cdot)&\leq d^{2}\Lambda |\deri^{2}u|+\Lambda[1+dC_{3}]+C_{\nu}[C_{3}+2C_{1}].
\end{align}
Then,
\begin{align}\label{p.10.0}
-2\psi_{\varepsilon}(\cdot)\langle\deri^{1}\xi,\deri^{1}u\rangle&\geq -2dC_{3}K_{7}[d^{2}\Lambda |\deri^{2}u|+\Lambda[1+dC_{3}]+C_{\nu}[C_{3}+2C_{1}]].
\end{align}
Since $[\partial_{k}u,\partial_{k}u]_{\inted}\geq0$, we get 
\begin{align}
\xi\psi'_{\varepsilon}(\cdot)\bigg[\sum_{k}[\partial_{k}u,\partial_{k}u]_{\inted}-\inted|\deri^{1}u|^{2}\bigg]&\geq-\xi\psi'_{\varepsilon}(\cdot)\inted|\deri^{1}u|^{2},
\end{align}
 {Meanwhile}, by (A1) and (A3), it follows  {that}
\begin{equation}
\begin{split}
&4\xi\psi''_{\varepsilon}(\cdot)\langle a\bar{\eta},\bar{\eta}\rangle\geq4\xi\psi''_{\varepsilon}(\cdot)\theta|\bar{\eta}|^{2}\geq0,\\
&\xi\psi'_{\varepsilon}(\cdot)\biggl[2\sum_{k}\langle a\deri^{1}\partial_{k}u,\deri^{1}\partial_{k}u\rangle-\tr[a\deri^{2}[ g^{2}]]\biggr]\geq\psi'_{\varepsilon}(\cdot)[\xi\theta|\deri^{2}u|^{2}-4d^{2}\Lambda^{3}].
\end{split}
\end{equation}
Since $\psi_{\varepsilon}(r)\leq\psi'_{\varepsilon}(r)r$, for all $r\in\R$, and from Lemmas \ref{Lb1}, \ref{gr1} and Remark \ref{R1},  we get 
\begin{equation}\label{p10}
\begin{split}
&\psi_{\varepsilon}(\cdot)\tr[a\deri^{2}\xi]\geq-\psi_{\varepsilon}(\cdot)d^{2}\Lambda K_{7}\geq-\psi'_{\varepsilon}(\cdot)d^{2}K_{7}\Lambda C_{3}^{2},\\
&2\psi'_{\varepsilon}(\cdot)\biggr[2\sum_{ij} a_{ij}\partial_{i}\xi\langle\deri^{1}u,\deri^{1}\partial_{j}u\rangle-\langle a\deri^{1}\xi,\deri^{1}[ g^{2}]\rangle\biggl]\geq-4\psi'_{\varepsilon}(\cdot)d^{3}\Lambda K_{7}[C_{3}|\deri^{2}u|+\Lambda^{2}].
\end{split}
\end{equation}
Applying \eqref{p9}, \eqref{p.9}, \eqref{p.10.0}--\eqref{p10} in \eqref{p8}, we obtain the inequality given in \eqref{gr3.2}, with $C_{5}=C_{5}(d,\Lambda,\nu,s,\alpha',K_{7})$. 
\end{proof}

\begin{lema}\label{cotaphi}
Let $\phi$ be as in Lemma \ref{gr3}. There exists a positive constant  $C_{6}$ independent of $\varepsilon$ such that $\phi\leq C_{6}$ on $B_{\beta' r}\subset\set$, for $r>0$ and $\beta'$ as in Remark \ref{R1}.
  \end{lema}

\begin{proof}
 Taking $x^{*}\in\overline{B}_{\beta'r}$ as a point where $\phi$ attains its maximum on $B_{\beta'r}$, it  suffices to bound $\phi(x^{*})$ by a constant independent of $\varepsilon$.  If $x^{*}\in\partial B_{\beta'r}$, then $\phi(x)\leq\phi(x^{*})=0$. Let  $x^{*}$ be in $B_{\beta'r}$. Observe, if  $|\deri^{1}u(x^{*})|^{2}-g(x^{*})^{2}<2\varepsilon$, from \eqref{p12.1}, we  see that $\phi(x)\leq\phi(x^{*})=\xi(x^{*})\psi_{\varepsilon}(|\deri^{1}u(x^{*})|^{2}-g(x^{*})^{2})\leq1$ on $B_{\beta'r}$. Therefore,  we obtain the result of Lemma \ref{cotaphi}. Assume that $|\deri^{1}u(x^{*})|^{2}-g(x^{*})^{2}\geq2\varepsilon$. Since $x^{*}\in B_{\beta'r}$, we know that
\begin{equation}\label{cphi1}
\begin{split}
&\deri^{1} \phi(x^{*})=0,\ \tr[a(x^{*})\deri^{2}\phi(x^{*})]\leq0,\\ &\phi(x^{*}+z)-\phi(x^{*})\leq0, \ \text{for $z$ with $x^{*}+z\in \overline{\set}_{\inted}$}.
\end{split}
\end{equation}
Then, evaluating $x^{*}$ in \eqref{gr3.2} and  by $\psi'_{\varepsilon}(\cdot)=1/\varepsilon$ at $x^{*}$; see \eqref{p12.1}, we get
\begin{align}\label{cotI1}
0&\geq\frac{1}{\varepsilon}[\theta\xi|\deri^{2}u|^{2}-C_{5}|\deri^{2}u|-C_{5}]-\frac{\xi}{\varepsilon}\inted|\deri^{1}u|^{2},\ \text{at}\ x^{*},
\end{align}
where $C_{5}$ as in Lemma \ref{gr3}.  {Meanwhile}, note that 
\begin{equation}\label{cotI3}
-\frac{\xi}{\varepsilon}\inted|\deri^{1}u|^{2}=-\frac{1}{\varepsilon}[\xi\inted_{\mathcal{B}}|\deri^{1}u|^{2}+\xi\inted_{\mathcal{B}^{\comp}}|\deri^{1}u|^{2}],\ \text{at}\ x^{*},
\end{equation}
where $\inted_{\mathcal{C}}|\deri^{1}u|^{2}\eqdef\int_{\mathcal{C}}[|\deri^{1}u(\cdot+z)|^{2}-|\deri^{1}u|^{2}]s(\cdot,z)\nu(\der z)$, with $\mathcal{C}\subseteq\R^{d}_{*}$,  and $\mathcal{B}\eqdef\{z\in\R^{d}_{*}:|\deri^{1}u(\cdot+z)|^{2}-[g(\cdot+z)]^{2}\leq|\deri^{1}u|^{2}- g^{2},\ \text{at}\ x^{*}\}$. By Remark \ref{s0} and Lemma \ref{g1}.i, the operator $\inted_{\mathcal{C}}$ is well defined. Now, from (A1), (A4) and Mean Value Theorem, it yields 
\begin{multline}	\label{cotI4}
-\frac{\xi}{\varepsilon}\inted_{\mathcal{B}}|\deri^{1}u|^{2}\geq-\frac{1}{\varepsilon}\bigg[\int_{\mathcal{B}\cap\{|z|\in(0,1)\}}\bigg[\int_{0}^{1}|\deri^{1}[[g(\cdot+tz)]^{2}]|\der t\bigg]|z|s(\cdot,z)\nu(\der z)\\
+\int_{\mathcal{B}\cap\{|z|\geq1\}}|[g(\cdot+z)]^{2}- g^{2}|s(\cdot,z)\nu(\der z)\bigg] \geq-\frac{2C_{\nu}\Lambda^{2}}{\varepsilon}\big[1+d^{\frac{1}{2}}\big],\ \text{at}\  x^{*},
\end{multline}
with $\Lambda,C_{\nu}$ as in (A1), \eqref{H4}, respectively. By Lemma \ref{g1}.ii, it follows that
\begin{equation}
-\frac{\xi}{\varepsilon}\inted_{\mathcal{B}^{\comp}}|\deri^{1}u|^{2}
=\frac{1}{\varepsilon}\biggl[-\inte_{\mathcal{B}^{\comp}}[\xi|\deri^{1}u|^{2}]+|\deri^{1}u|^{2}\inte_{\mathcal{B}^{\comp}}\xi+\sum_{i}[\xi,[\partial_{i}u]^{2}]_{\inted_{\mathcal{B}^{\comp}}}
\biggr],\ \text{at}\ x^{*},
\end{equation} 
where $[\xi,[\partial_{i}u]^{2}]_{\inted_{\mathcal{B}^{\comp}}}
\eqdef\int_{\mathcal{B}^{\comp}}[\xi(\cdot+z)-\xi][[\partial_{i}u(\cdot+z)]^{2}-[\partial_{i}u]^{2}]s(\cdot,z)\nu(\der z)$. Proceeding in a similar way that in \eqref{cotI4}, and using Lemma \ref{gr1} and Remark \ref{R1}, it is easy to verify that 
\begin{equation}
|\deri^{1}u|^{2}\inte_{\mathcal{B}^{\comp}}\xi+\sum_{i}[\xi,[\partial_{i}u]^{2}]_{\inted_{\mathcal{B}^{\comp}}}\geq-3K_{7}C_{\nu}C_{3}^{2}\big[2+d^{\frac{1}{2}}\big],\  \text{at}\ x^{*}. 
\end{equation} 
If $z\in\mathcal{B}^{\comp}$, then, from \eqref{p12.1} and \eqref{cphi1}, it can be verified that 
\begin{equation*}
\frac{1}{\varepsilon}[\xi(\cdot+z)|\deri^{1}u(\cdot+z)|^{2}-\xi|\deri^{2}u|^{2}]\leq\frac{1}{\varepsilon}[\xi(\cdot+z) [g(\cdot+z)]^{2}-\xi  g^{2}]-[\xi(\cdot+z)-\xi],\ \text{at}\ x^{*}.  
\end{equation*}
Then, proceeding as before,
\begin{align}\label{cotI2}
-\frac{1}{\varepsilon}\inte_{\mathcal{B}^{\comp}}[\xi|\deri^{1}u|^{2}]\geq-\frac{\Lambda^{2}K_{7}C_{\nu}}{\varepsilon}\big[2+3d^{\frac{1}{2}}\big]-K_{7}C_{\nu}\big[2+d^{\frac{1}{2}}\big],\ \text{at}\ x^{*}.
\end{align}
Applying \eqref{cotI3}--\eqref{cotI2} in \eqref{cotI1}, we obtain 
$0\geq \theta\xi|\deri^{2}u|^{2}-K_{8}|\deri^{2}u|-K_{8}[2+\varepsilon]$ at $x^{*}$,
for some $K_{8}=K_{8}(d,\Lambda,\nu,s,\alpha')$. Then, $[|\deri^{2}u(x^{*})|-K_{9}][|\deri^{2}u(x^{*})|-K_{10}]\leq0$,
where
$K_{9}\eqdef\frac{K_{8}+[K_{8}^{2}+4\theta K_{8}\xi(x^{*})[2+\varepsilon]]^{\frac{1}{2}}}{2\theta\xi(x^{*})}$ and $K_{10}\eqdef\frac{K_{8}-[K_{8}^{2}+4\theta K_{8}\xi(x^{*})[2+\varepsilon]]^{\frac{1}{2}}}{2\theta\xi(x^{*})}$. Notice that $K_{10}<0<K_{9}$. 
This implies that
$|\deri^{2}u(x^{*})|\leq\frac{K_{8}+[K_{8}^{2}+12\theta K_{8}]^{\frac{1}{2}}}{2\theta\xi(x^{*})}$. From here and \eqref{p9.0}, 
\begin{align*}
\phi(x)&=\xi(x)\psi_{\varepsilon}(|\deri^{1}u^{\varepsilon}(x)|^{2}-g(x)^{2})\notag\\
&\leq\xi(x^{*})[d^{2}\Lambda|\deri^{2}u(x^{*})|+K_{11}]\leq\frac{d^{2}\Lambda[K_{8}+[K_{8}^{2}+12\theta K_{8}]^{\frac{1}{2}}]}{2\theta}+K_{11},
\end{align*}
with $K_{11}\eqdef\Lambda[1+dC_{3}]+C_{\nu}[C_{3}+2C_{1}]$.
We conclude that  $\phi\leq C_{6}$ on $B_{\beta'r}$ with some constant 
$C_{6}=C_{6}(d,\Lambda,\nu,s,\alpha',K_{7})$.
\end{proof}
From Lemmas \ref{Lb1}, \ref{gr1}, \ref{cotaphi},  the following estimate is obtained in $\Lp^{p}_{\loc}(\set)$.

\begin{lema}\label{lemafrontera3}
Let $p\in (1,\infty)$. There exists a positive constant $C_{7}$ independent of $\varepsilon$ such that $||\deri^{2}u^{\varepsilon}||_{\Lp^{p}(B_{\beta r})}\leq C_{7}$, for $\beta\in(0,1]$ and $r>0$.
\end{lema}

\begin{proof}
Taking $w\eqdef\xi u$, we obtain $||\deri^{2}u||_{\Lp^{p}(B_{\beta r})}\leq ||\deri^{2}w||_{\Lp^{p}(B_{\beta' r})}$, with $B_{\beta' r}\subset\set$, $p\in(1,\infty)$, $r>0$ and $\beta'$ as in Remark \ref{R1}.  {By calculating the} first and second derivatives of $w$ in $B_{\beta'r}$, $\partial_{i}w=u\partial_{i}\xi+\xi\partial_{i}u$,  $\partial_{ji}w=\partial_{j}u\partial_{i}\xi+u\partial_{ji}\xi+\partial_{i}u\partial_{j}\xi+\xi\partial_{ji}u$, and from  \eqref{p13.0}, we get
\begin{equation}\label{dirif1}
\dif w=f,\ \text{in}\ B_{\beta'r}, \quad\text{s.t.}\  
w=0,\ \text{on}\ \partial B_{\beta'r},
\end{equation}
where  $f\eqdef\xi[h+\inte u-\psi_{\varepsilon}(|\deri^{1} u|^{2}- g^{2})]-u[\tr[a\deri^{2}\xi]-\langle\deri^{1} \xi,b\rangle]-2\langle a\deri^{1}\xi,\deri^{1} u\rangle$. We know that for  the linear Dirichlet problem  \eqref{dirif1}  (see \cite[Lemma 3.1]{lenh}), $||\deri^{2}w||_{\Lp^{p}(B_{\beta' r})}\leq K_{12}||f||_{\Lp^{p}(B_{\beta' r})}$  for some  $K_{12}=K_{12}(d,\Lambda,p,\beta',r)$. Estimating the terms of $f$ with the norm $||\cdot||_{\Lp^{p}(B_{\beta' r})}$ and using (A1)--(A4) and Lemmas \ref{Lb1}, \ref{gr1}, \ref{cotaphi},  it follows that there exists $C_{7}=C_{7}(d,\Lambda,\nu,s,\alpha',p,\beta,r)$ such that $||\deri^{2}u||_{\Lp^{p}(B_{\beta r})}\leq ||\deri^{2}w||_{\Lp^{p}(B_{\beta' r})}\leq C_{7}$.
\end{proof}

By Lemmas \ref{Lb1}, \ref{gr1}, \ref{cotaphi} and \ref{lemafrontera3},  the following result  can be easily verified, and the proof is omitted.

\begin{lema}\label{lemafrontera4}
Let $p\in(1,\infty)$. There exists  a positive constant $C_{8}$ independent of $\varepsilon$ such that $||u^{\varepsilon}||_{\text{W}^{2,p}(B_{\beta r})}\leq C_{8}$, for $\beta\in(0,1]$ and $r>0$.
\end{lema}
\subsection{Proof of Theorem \ref{princ1.0.1}}\label{proofHJB1} 

This subsection is devoted to proving Theorem \ref{princ1.0.1}. Let $p\in(1,\infty)$ be fixed, by Lemmas \ref{Lb1}, \ref{gr1}, and \ref{cotaphi}--\ref{lemafrontera4}, we obtain that for each open ball $B_{\beta r}\subset\set$,  $\beta\in(0,1]$ and $r>0$, there exist positive constants $C_{9},C_{10}$ independent of $\varepsilon$ such that 
\begin{align}
||u^{\varepsilon}||_{\sob^{1,\infty}(\set)}&<C_{9}\quad\text{and}\quad ||u^{\varepsilon}||_{\sob^{2,p}(B_{\beta r})}<C_{10}.\label{sob1}
\end{align}
Taking $p\in(d,\infty)$ fixed, from \eqref{sob1} and the Sobolev embedding  Theorem, we have that for each open ball $B_{\beta r}\subset \set$, there exists a positive constant $C_{11}$ independent of $\varepsilon$ such that
\begin{equation}\label{hol.1.1.0.1}
||u^{\varepsilon}||_{\hol^{1,\alpha}(\overline{B}_{\beta r})}\leq C_{11},\ \text{with}\ \alpha=1-\frac{d}{p}.
\end{equation}
 Using Arzel\`a-Ascoli Theorem, the reflexivity of $\Lp^{p}_{\loc}(\set)$; see \cite[Thm. 7.25, p. 158 and Thm. 2.46, p. 49, respectively]{rudin, adams},  and \eqref{sob1}--\eqref{hol.1.1.0.1}, we get that there exists a sub-sequence $\{u^{\varepsilon_{\kappa}}\}_{\kappa\geq1}$ of $\{u^{\varepsilon}\}_{\varepsilon\in(0,1)}$,  and $\tilde{u}\in\hol^{0,1}(\overline{\set})\cap\sob^{2,p}_{\loc}(\set)$ such that $u^{\varepsilon_{\kappa}}\underset{\varepsilon_{\kappa}\rightarrow0}{\longrightarrow}\tilde{u}$ in $\hol(\overline{\set})$, $\partial_{i}u^{\varepsilon_{\kappa}}\underset{\varepsilon_{\kappa}\rightarrow0}{\longrightarrow}\partial_{i}\tilde{u}$ $\text{in}\ \hol_{\loc}(\set)$, $\partial_{ij}u^{\varepsilon_{\kappa}}\underset{\varepsilon_{\kappa}\rightarrow0}{\longrightarrow}\partial_{ij}\tilde{u}$,  weakly $ \Lp^{p}_{\loc}(\set)$. Now, define
\begin{equation*}
u(x)\eqdef
\begin{cases}
\tilde{u}(x),&\text{if}\ x\in\set,\\
0,& \text{if}\  x\in\overline{\set}_{\inted}\setminus\set.
\end{cases}
\end{equation*}
It can be verified that $u$ is a continuous function on $\overline{\set}_{\inted}$, which satisfies $u\in\hol^{0,1}(\overline{\set})\cap\sob^{2,p}_{\loc}(\set)$ and 
\begin{equation}\label{conver.1.0}
\begin{split}
&u^{\varepsilon_{\kappa}}\underset{\varepsilon_{\kappa}\rightarrow0}{\longrightarrow} u,\ \text{in}\ \hol(\overline{\set}_{\inted}),\ \partial_{i}u^{\varepsilon_{\kappa}}\underset{\varepsilon_{\kappa}\rightarrow0}{\longrightarrow}\partial_{i} u,\ \text{in}\ \hol_{\loc}(\set),\\ &\partial_{ij}u^{\varepsilon_{\kappa}}\underset{\varepsilon_{\kappa}\rightarrow0}{\longrightarrow}\partial_{ij}u,\ \text{weakly  in}\ \Lp^{p}_{\loc}(\set).
\end{split}
\end{equation}
Since \eqref{p13.0} and \eqref{conver.1.0} hold, we only need to verify  \eqref{conver1.1}. Hence, we can conclude that the limit function $u$ is the solution to the HJB equation \eqref{p1}.

\begin{lema}
Let $\{u^{\varepsilon_{\kappa}}\}_{\kappa\geq0}$ and $u$ be the sub-sequence and the limit function that satisfy \eqref{conver.1.0}. Then,
\begin{equation}\label{conver1.1}
\int_{B_{r}}\varsigma\inted u^{\varepsilon_{\kappa}}\der x\underset{\varepsilon_{\kappa}\rightarrow0}{\longrightarrow}\int_{B_{r}}\varsigma\inted u\,\der x,\ \text{for any $\varsigma\in\hol^{\infty}_{\comp}(\set)$ with $\sop[\varsigma]\subset B_{r}\subset\set$.}
\end{equation} 
  
\end{lema}

\begin{proof}
Let $\varsigma\in \hol^{\infty}_{\comp}(\set)$ and $\sop[\varsigma]\subset B_{r}\subset\set$, with $0<r_{0}<\dist(\sop[\varsigma],\partial B_{r})\wedge1$. Then,
\begin{align}\label{conver1.2}
&\biggl|\int_{B_{r}}\varsigma(x)\inted [u^{\varepsilon_{\kappa}}-u](x)\der x\biggr|\notag\\
&\leq\int_{\sop[\varsigma]}|\varsigma(x)|\int_{\{|z|<r_{0}\}}\bigg[\int_{0}^{1}|\deri^{1}(u^{\varepsilon_{\kappa}}-u)(x+tz)|\der t\bigg]|z|s(x,z)\nu(\der z)\der x\notag\\
&\quad+\dfrac{1}{r_{0}}\int_{\sop[\varsigma]}|\varsigma(x)|\int_{\{r_{0}\leq|z|<1\}}|[u^{\varepsilon_{\kappa}}-u](x+z)-[u^{\varepsilon_{\kappa}}-u](x)|\,|z|s(x,z)\nu(\der z)\der x\notag\\
&\quad+\int_{\sop[\varsigma]}|\varsigma(x)|\int_{\{|z|\geq1\}}|[u^{\varepsilon_{\kappa}}-u](x+z)-[u^{\varepsilon_{\kappa}}-u](x)|s(x,z)\nu(\der z)\der x\notag\\
&\leq C_{\nu}||\varsigma||_{\Lp^{1}({B_{r}})}\bigg[||\deri^{1}[u^{\varepsilon_{\kappa}}-u]||_{\hol(B_{r})}+2\biggl[\frac{1}{r_{0}}+1\biggr]||u^{\varepsilon_{\kappa}}-u||_{\hol(\overline{\set}_{\inted})}\bigg].
\end{align}
From \eqref{conver.1.0} and letting $\varepsilon_{\kappa}\rightarrow0$ in \eqref{conver1.2}, it follows that \eqref{conver1.1}. With that, we finish the proof.
\end{proof}

We proceed to show the existence and uniqueness of the solution to the HJB equation \eqref{p1}.

\begin{proof}[Proof of Theorem \ref{princ1.0.1}. Existence]
	Let $p\in(d,\infty)$ be fixed, $\{u^{\varepsilon_{\kappa}}\}_{\kappa\geq0}$ and $u$ be the sub-sequence and the limit function, respectively, that satisfy \eqref{conver.1.0} and \eqref{conver1.1}. Recall that $u\in\hol^{0,1}(\overline{\set})\cap\sob^{2,p}_{\loc}(\set)$ and $u^{\varepsilon_{\kappa}}\in\hol^{3,\alpha'}(\overline{\set})$ is the unique solution to the NPIDD problem \eqref{p13.0} when $\varepsilon=\varepsilon_{\kappa}$. Then,  \eqref{p13.0}, \eqref{conver.1.0} and \eqref{conver1.1} imply $\int_{B_{ r}}\varsigma\Gamma u\der x\leq \int_{B_{ r}}\varsigma h\der x$ for each non-negative function $\varsigma$ in $\hol^{\infty}_{\comp}(B_{r})$, where $\sop[\varsigma]\subset B_{r}\subset \set $. From here, it follows that
	$\Gamma u\leq h$ a.e. in $\set$.  { {Meanwhile}, since $\psi_{\varepsilon}(|\deri^{1} u^{\varepsilon_{\kappa}}|^{2}- g^{2})$ is locally bounded (uniformly in $\varepsilon$); see Lemma \ref{cotaphi},  it follows that for each $x \in\set$, there exists an $\varepsilon' $ such that for all $\varepsilon_{\kappa}\leq\varepsilon' $,   $|\deri^{1} u^{\varepsilon_{\kappa}}(x)|\leq g(x)$. From here and that $|\deri^{1}[u^{\varepsilon_{\kappa}}-u](x)|\underset{\varepsilon_{\kappa}\rightarrow0}{\longrightarrow}0$, it yields    $|\deri^{1} u|\leq g$ in $\set$.} Suppose that $|\deri^{1} u(x^{*})|<g(x^{*})$, for some $x^{*}\in \set $. Then, by the continuity of  $\deri^{1} u$, there exists a small open ball $B_{r}\subset \set $ such that $x^{*}\in B_{r}$ and    
	$|\deri^{1} u|<g$ in $ B_{r}$. Since $||\deri^{1}[ u^{\varepsilon_{\kappa}}-u]||_{\hol(\overline{B_{r}})}\underset{\varepsilon_{\kappa}\rightarrow0}{\longrightarrow}0$, we obtain that there exists $\varepsilon_{\kappa_{0}}$ such that for each $\varepsilon_{\kappa}\leq \varepsilon_{\kappa_{0}}$,  $|\deri^{1} u^{\varepsilon_{\kappa}}|<g$ in $B_{r}$. Then, from \eqref{p13.0} and the definition of $\psi_{\varepsilon}$, it follows that for each $\varepsilon_{\kappa}\leq \varepsilon_{\kappa_{0}}$,  $\Gamma u^{\varepsilon_{\kappa}}= h$ in $B_{r}$. Then, $\int_{B_{  r}}[\Gamma u^{\varepsilon_{\kappa}}]\varsigma\der x= \int_{B_{ r}}h\varsigma\der x$, for any non-negative function $\varsigma$ in $\hol^{\infty}_{\comp}(B_{ r})$, with $\sop[\varsigma]\subset B_{r}\subset \set $. From here and  using \eqref{conver.1.0}--\eqref{conver1.1}, we get  $\int_{B_{ r}}\varsigma \Gamma u \der x= \int_{B_{ r}}\varsigma h\der x$. Therefore,
	$\Gamma u= h$, a.e. in $B_{  r}$. 
	By the arguments seen previously, we conclude that $u$ is a solution to the HJB equation  \eqref{p1}  a.e. in $\set $.
\end{proof}

\begin{proof}[Proof of Theorem \ref{princ1.0.1}. Uniqueness]
Let  $p\in(d,\infty)$ be fixed. Suppose there exist $u_{1},u_{2}\in\hol^{0,1}(\overline{\set })\cap\sob^{2,p}_{\loc}(\set )$, two solutions to the HJB equation  \eqref{p1}. Let $x^{*}\in\overline{\set }$ be the point where $u_{1}-u_{2}$ attains its maximum. If $x^{*}\in\partial \set $, it is easy to see 
$[u_{1}-u_{2}](x)\leq[u_{1}-u_{2}](x^{*})=0$ for all $x\in\overline{\set}$. Let us assume that $x^{*}\in \set$. In this case one wishes to prove that $[u_{1}-u_{2}](x^{*})\leq0$, which we demonstrate  by contradiction. Suppose $[u_{1}-u_{2}](x^{*})>0$ and take $f\eqdef[1-\rho]u_{1}-u_{2}$ on $\overline{\set}$  such that $f(x^{*})>0$, for some $\rho>0$ small enough. Using $f=0$ on $\overline{\set}_{\inted}\setminus\set   $, it follows that $f(x^{*}_{1})>0$, where  $x^{*}_{1}\in \set $ is the point  where  $f$ attains its maximum. Besides, we have
$\deri^{1}f(x^{*}_{1})=[1-\rho]\deri^{1}u_{1}(x^{*}_{1})-\deri^{1}u_{2}(x^{*}_{1})=0$ and $f(x^{*}_{1}+z)\leq f(x^{*}_{1})$, for $z\in\R^{d}_{*}$ with $x^{*}_{1}+z\in \overline{\set}_{\inted}$. Then, $\inted f(x^{*}_{1})\leq 0$. Since $\deri^{1}f(x^{*}_{1})=0$, $|\deri^{1}u_{1}(x^{*}_{1})|\leq g(x^{*}_{1})$ and $1-\rho<1$, we get  $|\deri^{1}u_{2}(x^{*}_{1})|=[1-\rho]|\deri^{1}u_{1}(x^{*}_{1})|<g(x^{*}_{1})$. This implies that  there exists $\mathcal{V}_{x^{*}_{1}}$ a neighborhood of $x_{1}^{*}$ such that  $\Gamma u_{2}=h$ and $\Gamma u_{1}\leq h$ in $\mathcal{V}_{x_{1}^{*}}$. Then, $\Gamma f\leq-\rho h$  in $\mathcal{V}_{x_{1}^{*}}$, and hence $\tr[a\deri^{2}f]\geq \langle b,\deri^{1}f\rangle+cf-\inted f+\rho h,\ \text{in}\ \mathcal{V}_{x_{1}^{*}}$.  {By} using Bony's maximum principle (see \cite{lions}), it yields 
$$0\geq\limess_{x\rightarrow x_{1}^{*}}\tr[a(x)\deri^{2}f(x)]\geq c(x^{*}_{1})f(x_{1}^{*})-\inted f(x_{1}^{*})+\rho h(x_{1}^{*}),$$
which is a contradiction since $c(x_{1}^{*})f(x_{1}^{*})>0$, $-\inted f(x_{1}^{*})\geq0$ and $\rho h(x^{*}_{1})\geq0$. The application of Bony's maximum principle is permitted here because $u_{1},u_{2}\in \sob^{2,p}_{\loc}(\set )$ and $d< p<\infty$. Therefore, it yields $[u_{1}-u_{2}](x)\leq[u_{1}-u_{2}](x^{*})\leq0$ for all $x\in\overline{\set}$. Taking $u_{2}-u_{1}$ and proceeding in the same way as before, it follows that $u_{2}-u_{1}\leq 0$  in $ \set $, and hence  we conclude that the solution $u$ to the HJB equation \eqref{p1} is unique.
\end{proof}

\section{Penalized control problem and proof of Proposition \ref{veri1}}\label{HJBPro} 

This section is devoted to verifying that the value function $V$ and $u$ agree on $\set$, which are the value function defined in \eqref{vf1} and the solution to the HJB equation \eqref{esd5}, respectively. For this purpose, we   introduce a class of penalized controls that belong to $\mathcal{U}$. Recall that $\mathcal{U}$ is the set of admissible controls $(n,\zeta)$ that satisfy \eqref{cont.1}. Take the penalized controls set $\mathcal{U}^{\varepsilon}$ by
\begin{equation*}
\mathcal{U}^{\varepsilon}\eqdef\{(n,\zeta)\in\mathcal{U}: \zeta_{\mathpzc{t}}\ \text{is absolutely continuous,}\ 0\leq\dot{\zeta}_{\mathpzc{t}}\leq 2C_{3}/\varepsilon\},\ \text{with}\ \varepsilon\in(0,1)\ \text{fixed}, 
\end{equation*} 
where $C_{3}$ is a positive constant as in Lemma \ref{gr1}, which is independent of $\varepsilon$. Then, for each $(n,\zeta)\in\mathcal{U}^{\varepsilon}$ and $\tilde{x}\in\set$,  the process $X^{n,\zeta}=\{X^{n,\zeta}_{\mathpzc{t}}:\mathpzc{t}\geq0\}$ evolves in the following way
\begin{equation}\label{esd3.1.0.0}
 X^{n,\zeta}_{\mathpzc{t}}=\tilde{x}-\int_{0}^{\mathpzc{t}}[ \tilde{b} (X^{n,\zeta}_{\mathpzc{s}})+n_{\mathpzc{s}}\dot{\zeta}_{\mathpzc{s}}]\der \mathpzc{s}+\int_{0}^{\mathpzc{t}}\sigma(X^{n,\zeta}_{\mathpzc{s}})\der  W_{\mathpzc{s}}+\int_{0}^{\mathpzc{t}}\der  J_{\mathpzc{s}},\ \text{with}\ \mathpzc{t}\geq 0, 
\end{equation}
where $W$ is a $d$-dimensional SBM as in Subsection \ref{Prob.1} and $ J$ is the jump process given by \eqref{esd2.1}.  Notice that $\Delta X^{n,\zeta}=\Delta  J$.  Recall that here $ \tilde{b}:\R^{d}\longrightarrow\R^{d}$, $\sigma:\R^{d}\longrightarrow\R^{d\times d}$,  and $s:\overline{\set}\times\R^{d}\longrightarrow[0,1]$ satisfy (A1)--(A5). Then, the SDE \eqref{esd3.1.0.0} has a unique c\`adl\`ag adapted solution $X^{n,\zeta}$; see \cite{dade}. The penalized cost related to this class of controls is defined by 
\begin{equation*}
\mathcal{V}_{n,\zeta}(\tilde{x})=\E_{\tilde{x}}\biggl[\int_{0}^{\tau^{n,\zeta}}\expo^{-q\mathpzc{s}}[  h (X^{n,\zeta}_{\mathpzc{s}})+l_{\varepsilon}(X^{n,\zeta}_{\mathpzc{s}},\dot{\zeta}_{\mathpzc{s}}n_{\mathpzc{s}})]\der \mathpzc{s}\biggr],\ \text{for }\ (n,\zeta)\in\mathcal{U}^{\varepsilon},
\end{equation*}
where $\tau^{n,\zeta}\eqdef\inf\{ \mathpzc{t}>0:X_{\mathpzc{t}}^{n,\zeta}\notin\set\}$,  $ h :\R^{d}\longrightarrow\R$ is continuous and non-negative, and $l_{\varepsilon}(x,y)\eqdef\sup_{\gamma\in\R^{d}}\{\langle \gamma,y \rangle-H_{\varepsilon}(x,\gamma)\}$ is the  Legendre transform of $H_{\varepsilon}(x,\gamma)\eqdef\psi_{\varepsilon}(|\gamma|^{2}-g(x)^{2})$, where $ g:\R^{d}\longrightarrow\R$ is continuous and non-negative. Notice that, for each $x\in\R^{d}$ fixed,  $H_{\varepsilon}(x,\gamma)$ is a $\hol^{2}$ and convex function with respect to the variable $\gamma\in\R^{d}$, since $\psi_{\varepsilon}\in\hol^{\infty}(\R)$ is convex function; see \eqref{p12.1}. The value function for this problem is given by
\begin{equation}\label{Vfp1}
V^{\varepsilon}(\tilde{x})\eqdef\inf_{(n,\zeta)\in\mathcal{U}^{\varepsilon}}\mathcal{V}_{n,\zeta}(\tilde{x}).
\end{equation}
A heuristic derivation from dynamic programming principle (see \cite[Ch. VIII]{flem}) shows that the NPIDD problem corresponding to the value function $V^{\varepsilon}$ is of the form
\begin{equation}\label{p13.0.1.0}
[q-\Gamma_{1}]  u^{\varepsilon}+ \sup_{y\in\R^{d}}\{\langle \deri^{1}u^{\varepsilon},y\rangle-l_{\varepsilon}(\cdot,y)\}= h ,\ \text{in}\ \set,\quad\text{s.t.}\ 
   u^{\varepsilon}=0,\ \text{on}\ \overline{\set}_{\inted}\setminus\set ,
\end{equation}
where $\Gamma_{1}$ is as in \eqref{esd3.0}.  Since $H_{\varepsilon}(x,\gamma)$ is $\hol^{2}$ with respect to the variable $\gamma$, it follows that $H_{\varepsilon}(x,\gamma)=\sup_{y\in\R^{d}}\{\langle \gamma,y\rangle-l_{\varepsilon}(x,y)\}$. Then, the NPIDD problem \eqref{p13.0.1.0} can be written as
\begin{equation}\label{NPIDD.1}
[q-\Gamma_{1}]  u^{\varepsilon}+ \psi_{\varepsilon}(|\deri^{1} u^{\varepsilon}|^{2}- g   ^{2})=  h  ,\ \text{in}\ \set,\quad\text{s.t.}\ 
   u^{\varepsilon}=0,\ \text{on}\ \overline{\set}_{\inted}\setminus\set .
\end{equation}
Assuming from now on that  $a_{ij}=\frac{1}{2}(\sigma\sigma^{\trans})_{ij}$, $ b _{i}$, $ h $, $ g   $, $  s  $ satisfy (A1)--(A5), an immediate consequence of Proposition \ref{princ1.0} is the following corollary.

\begin{coro}\label{NPIDD.1.2}
The NPIDD problem \eqref{NPIDD.1} has a unique  non-negative solution $u^{\varepsilon}$ in $\hol^{3,\alpha'}(\overline{\set})$, for each $\varepsilon\in(0,1)$. 
\end{coro}

\begin{rem}
 Without loss of generality  we assume that $\psi_{\varepsilon}$ is non-decreasing as $\varepsilon\downarrow0$; see \cite{zhu}.
\end{rem}

\begin{coro}\label{nc1}
 Let $u^{\varepsilon}$ be the unique non-negative solution to the NPIDD problem, for each $\varepsilon\in(0,1)$. Then, $u^{\varepsilon}$ is non-increasing as $\varepsilon\downarrow 0$.
\end{coro}

\begin{proof}
Let $u^{\varepsilon_{1}}, u^{\varepsilon_{2}}$ be the unique solutions to the NPIDD problem \eqref{p13.0} when $\varepsilon=\varepsilon_{1},\varepsilon_{2}$, respectively, with $\varepsilon_{2}\leq\varepsilon_{1}$. Since $\psi_{\varepsilon_{2}}\geq\psi_{\varepsilon_{1}}$ and $u^{\varepsilon_{2}}$ is the unique solution to \eqref{p13.0} when $\varepsilon=\varepsilon_{2}$, we see that
\begin{equation*}
[q-\Gamma_{1}]u^{\varepsilon_{2}}+\psi_{\varepsilon_{1}}(|\deri^{1}u^{\varepsilon_{2}}|^{2}- g   ^{2})\leq h,\ \text{in}\ \set,\quad\text{s.t.}\ 
u^{\varepsilon_{2}}=0,\ \text{on}\ \overline{\set}_{\inted}\setminus\set .
\end{equation*}
From Lemma \ref{supsuper}, it follows that $u^{\varepsilon_{2}}\leq u^{\varepsilon_{1}}$ on $\overline{\set}$. Therefore, $u^{\varepsilon}$ is non-increasing as $\varepsilon\downarrow 0$.
\end{proof}
Now we construct our optimal stochastic control candidate $(n^{\varepsilon,*},\zeta^{\varepsilon,*})$ for the problem \eqref{Vfp1}.
Consider the following SDE
\begin{align}\label{esd3.1.0}
X^{\varepsilon,*}_{\mathpzc{t}\wedge\tau^{*}_{\varepsilon}}&=\tilde{x}+\int_{0}^{\mathpzc{t}\wedge\tau^{*}_{\varepsilon}}\sigma(X^{\varepsilon,*}_{\mathpzc{s}})\der  W_{\mathpzc{s}}+\int_{0}^{\mathpzc{t}\wedge\tau^{*}_{\varepsilon}}\der  J_{\mathpzc{s}}\notag\\
&\quad-\int_{0}^{\mathpzc{t}\wedge\tau^{*}_{\varepsilon}}[ \tilde{b} (X^{\varepsilon,*}_{\mathpzc{s}})+2\psi'_{\varepsilon}(|\deri^{1} u^{\varepsilon}(X^{\varepsilon,*}_{\mathpzc{s}})|^{2}- g   (X^{\varepsilon,*}_{\mathpzc{s}})^{2})\deri^{1}u^{\varepsilon}(X^{\varepsilon,*}_{\mathpzc{s}})]\der \mathpzc{s},
\end{align}
with $\tilde{x}\in \set$, $\mathpzc{t}\geq0$ and $\tau^{*}_{\varepsilon}\eqdef\inf\{\mathpzc{t}>0:X^{\varepsilon,*}_{\mathpzc{t}}\notin\set\}$. Observe that $\psi'_{\varepsilon}(|\deri^{1} u^{\varepsilon}|^{2}-g^{2})\deri^{1}u^{\varepsilon}$ satisfies \eqref{jum.1}, since it is a bounded Lipschitz continuous function on $\overline{\set}$. Then,  the SDE \eqref{esd3.1.0} has a unique c\`adl\`ag adapted solution $X^{\varepsilon,*}$; see \cite{KK2014}. Defining the  control process $(n^{\varepsilon,*},\zeta^{\varepsilon,*})$ by   
\begin{equation}\label{opt.1}
n^{\varepsilon,*}_{\mathpzc{t}}=
\begin{cases}
\frac{\deri^{1}u^{\varepsilon}(X^{\varepsilon,*}_{\mathpzc{t}})}{|\deri^{1}u^{\varepsilon}(X^{\varepsilon,*}_{\mathpzc{t}})|}, &\text{if}\ |\deri^{1}u^{\varepsilon}(X^{\varepsilon,*}_{\mathpzc{t}})|\neq0,\\
\gamma_{0},& \text{if}\ |\deri^{1}u^{\varepsilon}(X^{\varepsilon,*}_{\mathpzc{t}})|=0, 
\end{cases}
\end{equation}
with $\gamma_{0}\in\R^{d}$ a unit vector fixed, and $\zeta^{\varepsilon,*}_{\mathpzc{t}}=\int_{0}^{\mathpzc{t}}\dot{\zeta}^{\varepsilon,*}_{\mathpzc{s}}\der \mathpzc{s}$, with
\begin{align}\label{opt.2}
\dot{\zeta}^{\varepsilon,*}_{\mathpzc{s}}= 2\psi'_{\varepsilon}(|\deri^{1} u^{\varepsilon}(X^{\varepsilon,*}_{\mathpzc{s}})|^{2}- g   (X^{\varepsilon,*}_{\mathpzc{s}})^{2})|\deri^{1} u^{\varepsilon}(X^{\varepsilon,*}_{\mathpzc{s}})|,
\end{align}
we see that for  $\mathpzc{t}\in[0,\tau^{*}_{\varepsilon})$, $n^{\varepsilon,*}_{\mathpzc{t}}\dot{\zeta}^{\varepsilon,*}_{\mathpzc{t}}=2\psi'_{\varepsilon}(|\deri^{1} u^{\varepsilon}(X^{\varepsilon,*}_{\mathpzc{t}})|^{2}- g   (X^{\varepsilon,*}_{\mathpzc{t}})^{2})\deri^{1}u^{\varepsilon}(X^{\varepsilon,*}_{\mathpzc{t}})$, $\Delta\zeta^{\varepsilon,*}_{\mathpzc{t}}=0$, $|n^{\varepsilon,*}_{\mathpzc{t}}|=1$ and, by \eqref{p12.1} and Lemma \ref{gr1},  $\dot{\zeta}^{\varepsilon,*}_{\mathpzc{t}}\leq\frac{2C_{3}}{\varepsilon}$. On the event $\{\tau^{*}_{\varepsilon}=\infty\}$, the control process $(n^{\varepsilon,*},\zeta^{\varepsilon,*})$ belongs to $\mathcal{U}^{\varepsilon}$. On the event $\{\tau^{*}_{\varepsilon}<\infty \}$, since $u^{\varepsilon}\in\hol^{3,\alpha'}(\overline{\set})$, $u^{\varepsilon}=0$ on $\overline{\set}_{\inted}\setminus\set $, and $X^{\varepsilon,*}_{\tau^{*}_{\varepsilon}}\in\overline{\set}_{\inted}\setminus\set $,  we take $\dot{\zeta}^{\varepsilon,*}_{\mathpzc{t}}\equiv 0$ and  $n^{\varepsilon,*}_{\mathpzc{t}}\eqdef\gamma_{0}$, for $\mathpzc{t}>\tau^{*}_{\varepsilon}$. In this way, we have that $(n^{\varepsilon,*},\zeta^{\varepsilon,*})\in\mathcal{U}^{\varepsilon}$. 

\begin{lema}[Verification Lemma for penalized control problem]\label{convexu1.0}
Let  $\varepsilon\in(0,1)$ be fixed. Then,
\begin{enumerate}
\item[(i)] For each $(n,\zeta)\in\mathcal{U}^{\varepsilon}$, $u^{\varepsilon}\leq\mathcal{V}_{n,\zeta}$ on $\overline{\set}$. 

\item[(ii)] Let $X^{\varepsilon,*}, (n^{\varepsilon,*},\zeta^{\varepsilon,*})$ be the solution process to the SDE \eqref{esd3.1.0} and the control process given by \eqref{opt.1}--\eqref{opt.2}, respectively. Then, $u^{\varepsilon}=\mathcal{V}_{n^{\varepsilon,*},\zeta^{\varepsilon,*}}=V^{\varepsilon}$ on $\overline{\set}$.
\end{enumerate}
\end{lema}
From now on, for simplicity of notation, we replace $X^{n,\zeta}$  by $X$ in the proofs of the results.
\begin{proof}[Proof of Lemma \ref{convexu1.0}]
Let $\varepsilon\in(0,1)$  be fixed, $X=\{X_{\mathpzc{t}}:\mathpzc{t}\geq0\}$ be the process which evolves  as in \eqref{esd3.1.0.0}, with $(n,\zeta)\in\mathcal{U}^{\varepsilon}$ and $\tilde{x}\in \set$ an initial state.  Notice that $u^{\varepsilon}$ is in $\hol^{3}(\overline{\set})$. Then, integration by parts and It\^o's formula imply (see \cite[Cor. 2 and Thm. 33, pp. 68 and 81, respectively]{pro}) 
\begin{align}\label{expand.1}
u^{\varepsilon}(\tilde{x})&=
\expo^{-q[\mathpzc{t}\wedge\tau]}u^{\varepsilon}(X_{\mathpzc{t}\wedge\tau})-\int_{0}^{\mathpzc{t}\wedge\tau}\expo^{-q\mathpzc{s}}\langle\deri^{1}u^{\varepsilon}(X_{\mathpzc{s}}),\der  J_{\mathpzc{s}}\rangle\notag\\
&\quad+\int_{0}^{\mathpzc{t}\wedge\tau}\expo^{-q\mathpzc{s}}[qu^{\varepsilon}(X_{\mathpzc{s}})+\langle \deri^{1}u^{\varepsilon}(X_{\mathpzc{s}}),\tilde{b}(X_{\mathpzc{s}})\rangle-\tr[a(X_{\mathpzc{s}})\deri^{2}u^{\varepsilon}(X_{\mathpzc{s}})]]\der \mathpzc{s}\notag\\
&\quad-\int_{0}^{\mathpzc{t}\wedge\tau}\expo^{-q\mathpzc{s}}\langle\deri^{1} u^{\varepsilon}(X_{\mathpzc{s}}),\sigma(X_{\mathpzc{s}})\der  W_{\mathpzc{s}}\rangle+\int_{0}^{\mathpzc{t}\wedge\tau}\expo^{-q\mathpzc{s}}\langle\deri^{1} u^{\varepsilon}(X_{\mathpzc{s}-}),n_{\mathpzc{s}}\rangle\der\zeta_{\mathpzc{s}}\notag\\
&\quad-\sum_{0\leq\mathpzc{s}\leq\mathpzc{t}\wedge\tau}\expo^{-q\mathpzc{s}}[u^{\varepsilon}(X_{\mathpzc{s}-}+\Delta X_{\mathpzc{s}})-u^{\varepsilon}(X_{\mathpzc{s}-})-\langle\deri^{1}u^{\varepsilon}(X_{\mathpzc{s}-}),\Delta X_{\mathpzc{s}}\rangle],
\end{align}
where $a_{ij}=\frac{1}{2}(\sigma\sigma^{\trans})_{ij}$ and $\tau=\inf\{ \mathpzc{t}>0:X_{\mathpzc{t}}\notin\set\}$. 
 {Meanwhile}, since $\Delta\zeta\equiv0$, it can be verified 
\begin{align}\label{expand.1.2}
&-\sum_{0\leq\mathpzc{s}\leq\mathpzc{t}\wedge\tau}\expo^{-q\mathpzc{s}}[u^{\varepsilon}(X_{\mathpzc{s}-}+\Delta X_{\mathpzc{s}})-u^{\varepsilon}(X_{\mathpzc{s}-})-\langle\deri^{1}u^{\varepsilon}(X_{\mathpzc{s}-}),\Delta X_{\mathpzc{s}}\rangle]\notag\\
&\quad=-\int_{0}^{\mathpzc{t}\wedge\tau}\int_{\mathcal{S}}\expo^{-q\mathpzc{s}}[u^{\varepsilon}(X_{\mathpzc{s}-}+z)-u^{\varepsilon}(X_{\mathpzc{s}-})-\langle\deri^{1}u^{\varepsilon}(X_{\mathpzc{s}-}),z\rangle]\uno_{\{\rho\in[0,s(X_{\mathpzc{s}},z)]\}}N(\der\rho,\der z,\der\mathpzc{s})\notag\\
&\quad=-\widetilde{\mathcal{M}}^{\varepsilon}_{\mathpzc{t}\wedge\tau}+\int_{0}^{\mathpzc{t}\wedge\tau}\expo^{-q\mathpzc{s}}\langle\deri^{1}u^{\varepsilon}(X_{\mathpzc{s}-}),\der  J_{\mathpzc{s}}\rangle-\int_{0}^{\mathpzc{t}\wedge\tau}\expo^{-q\mathpzc{s}}\inte u^{\varepsilon}(X_{\mathpzc{s}})\der\mathpzc{s}\notag\\
&\qquad+\int_{0}^{\mathpzc{t}\wedge\tau}\int_{\{|z|\in(0,1)\}}\expo^{-q\mathpzc{s}}\langle\deri^{1}u^{\varepsilon}(X_{\mathpzc{s}}),z\rangle s(X_{\mathpzc{s}},z)\nu(\der z)\der\mathpzc{s},
\end{align}
where
\begin{equation}\label{M.1}
\widetilde{\mathcal{M}}^{\varepsilon}_{\mathpzc{t}\wedge\tau}\eqdef\int_{0}^{\mathpzc{t}\wedge\tau}\int_{\mathcal{S}}\expo^{-q\mathpzc{s}}[u^{\varepsilon}(X_{\mathpzc{s}-}+z)-u^{\varepsilon}(X_{\mathpzc{s}-})] \uno_{\{\rho\in[0,s (X_{\mathpzc{s}-},z)]  \}}\widetilde{N}(\der \rho,\der z,\der \mathpzc{s}).
\end{equation}
Recall that $\mathcal{S}=[0,1]\times\R^{d}$ and  $\widetilde{N}(\der\rho,\der z,\der \mathpzc{t})= N(\der\rho,\der z,\der \mathpzc{t})-\eta(\der \rho,\der z)\der \mathpzc{t}$ is the  compensated Poisson random  measure with intensity $\eta(\der\rho,\der z)\der\mathpzc{t}=\der\rho\nu(\der z)\der\mathpzc{t}$. Then, from \eqref{expand.1}, \eqref{expand.1.2} and noting that $\der\zeta_{\mathpzc{s}}=\dot{\zeta_{\mathpzc{s}}}\der\mathpzc{s}$ and 
\begin{align*}
\int_{0}^{\mathpzc{t}\wedge\tau}\expo^{-q\mathpzc{s}}\langle\deri^{1}u^{\varepsilon}(X_{\mathpzc{s}}), b(X_{\mathpzc{s}})\rangle\der\mathpzc{s} &= \int_{0}^{\mathpzc{t}\wedge\tau}\expo^{-q\mathpzc{s}}\bigg[\langle \deri^{1}u^{\varepsilon}(X_{\mathpzc{s}}),\tilde{b}(X_{\mathpzc{s}})\rangle\notag\\ &\quad+\int_{\{|z|\in(0,1)\}}\langle\deri^{1}u^{\varepsilon}(X_{\mathpzc{s}}),z\rangle s(X_{\mathpzc{s}},z)\nu(\der z)\bigg]\der\mathpzc{s},
\end{align*}
it follows that
\begin{align}\label{expand1.0.0}
u^{\varepsilon}(\tilde{x})&= 
\expo^{-q[\mathpzc{t}\wedge\tau]} u^{\varepsilon}(X_{\mathpzc{t}\wedge\tau})-\mathcal{M}^{\varepsilon}_{\mathpzc{t}\wedge\tau}+\int_{0}^{\mathpzc{t}\wedge\tau}\expo^{-q\mathpzc{s}} [[q-\Gamma_{1}] u^{\varepsilon}(X_{\mathpzc{s}})+\langle\deri^{1} u^{\varepsilon}(X_{\mathpzc{s}}),\dot{\zeta}_{\mathpzc{s}}n_{\mathpzc{s}}\rangle]\der \mathpzc{s},
\end{align}
with $\Gamma_{1}$ as in \eqref{esd3.0},   
\begin{equation}\label{M.2}
\mathcal{M}^{\varepsilon}_{\mathpzc{t}\wedge\tau}\eqdef\widetilde{\mathcal{M}}^{\varepsilon}_{\mathpzc{t}\wedge\tau}+\overline{\mathcal{M}}^{\varepsilon}_{\mathpzc{t}\wedge\tau}\quad\text{and}\quad
\overline{\mathcal{M}}^{\varepsilon}_{\mathpzc{t}\wedge\tau}\eqdef\int_{0}^{\mathpzc{t}\wedge\tau}\expo^{-q\mathpzc{s}}\langle\deri^{1} u^{\varepsilon}(X_{\mathpzc{s}}),\sigma(X_{\mathpzc{s}})\der  W_{\mathpzc{s}}\rangle.
\end{equation}
Observe that $\E_{\tilde{x}}[\int_{0}^{t}\int_{\mathcal{S}}\tilde{f}(X_{\mathpzc{s}-},z)N(\der\rho,\der z,\der\mathpzc{s})]=\int_{0}^{t}\int_{\mathcal{S}}\E_{\tilde{x}}[\tilde{f}(X_{\mathpzc{s}-},z)]\eta(\der \rho,\der z)\der\mathpzc{s}$. From here and taking $\tilde{f}(X_{\mathpzc{s}-},z)=\expo^{-q\mathpzc{s}}[u^{\varepsilon}(X_{\mathpzc{s}-}+z)-u^{\varepsilon}(X_{\mathpzc{s}-})] \uno_{\{\rho\in[0, s (X_{\mathpzc{s}-},z)]\}}$, it can be verified that $\widetilde{\mathcal{M}}^{\varepsilon}=\{\widetilde{ \mathcal{M}}^{\varepsilon}_{\mathpzc{t}\wedge\tau}:\mathpzc{t}\geq0\}$ is a martingale. Moreover, $\widetilde{\mathcal{M}}^{\varepsilon}$ is square integrable, since
\begin{align*}
\E_{\tilde{x}}[[\widetilde{\mathcal{M}}^{\varepsilon}_{\mathpzc{t}\wedge\tau}]^{2}]&=\E_{\tilde{x}}\bigg[\int_{0}^{\mathpzc{t}\wedge\tau}\int_{\mathcal{S}}[\tilde{f}(X_{\mathpzc{s}-},z)]^{2}\eta(\der\rho,\der z)\der\mathpzc{s}\bigg]\notag\\
&=\E_{\tilde{x}}\bigg[\int_{0}^{\mathpzc{t}\wedge\tau}\int_{\R^{d}_{*}}\expo^{-2q\mathpzc{s}}[u^{\varepsilon}(X_{\mathpzc{s}-}+z)-u^{\varepsilon}(X_{\mathpzc{s}-})]^{2}s(X_{\mathpzc{s}-},z)\nu(\der z)\der\mathpzc{s}\bigg]\notag\\
&\leq\mathpzc{t}\bigg[C_{3}^{2}\int_{\{|z|\in(0,1)\}}|z|^{2}\nu(\der z)+4C_{1}^{2}\int_{\{|z|\geq1\}}\nu(\der z)\bigg]<\infty,\ \text{for}\ \mathpzc{t}\geq0\ \text{fixed}, 
\end{align*}
where $C_{1}$ and $C_{3}$ are as in Lemmas \ref{Lb1} and \ref{gr1}, respectively.  {Meanwhile}, It\^o's  isometry and the continuity of $\sigma$ and $u^{\varepsilon}$ on $\overline{\set}$, imply that 
\begin{align*}
\E_{\tilde{x}}[[\overline{\mathcal{M}}^{\varepsilon}_{\mathpzc{t}\wedge\tau}]^{2}]&\leq2\sum_{ij}\E_{\tilde{x}}\bigg[\int_{0}^{\mathpzc{t}\wedge\tau}[\expo^{-q\mathpzc{s}}\partial_{i}u^{\varepsilon}(X_{\mathpzc{s}-})\sigma_{ij}(X_{\mathpzc{s}-})]^{2}\der \mathpzc{s}\bigg]\leq 2\mathpzc{t}C_{3}^{2}\sum_{ij}\sup_{x\in\overline{\set}}\{\sigma_{ij}(x)^{2}\}<\infty,\ 
\end{align*}
for $\mathpzc{t}\geq0$ fixed. This implies that $\overline{\mathcal{M}}^{\varepsilon}=\{\overline{ \mathcal{M}}^{\varepsilon}_{\mathpzc{t}\wedge\tau}:\mathpzc{t}\geq0\}$ is a square integrable martingale. Therefore, the process $\mathcal{M}^{\varepsilon}=\{ \mathcal{M}^{\varepsilon}_{\mathpzc{t}\wedge\tau}:\mathpzc{t}\geq0\}$ is also a square integrable  martingale, with $\mathcal{M}^{\varepsilon}_{0}=0$.   Notice that, by Doob's stopping theorem, $\E_{\tilde{x}}[ \mathcal{M}^{\varepsilon}_{\mathpzc{t}\wedge\tau}]=\E_{\tilde{x}}[ \mathcal{M}^{\varepsilon}_{0}]=0$.    Taking the expected value in  \eqref{expand1.0.0}, it follows that
\begin{align}\label{exp.1}
u^{\varepsilon}(\tilde{x})&=\E_{\tilde{x}}\biggl[\expo^{-q[\mathpzc{t}\wedge\tau]} u^{\varepsilon}(X_{t\wedge\tau})+\int_{0}^{\mathpzc{t}\wedge\tau}\expo^{-q\mathpzc{s}} [[q-\Gamma_{1}] u^{\varepsilon}(X_{\mathpzc{s}})+\langle\deri^{1} u^{\varepsilon}(X_{\mathpzc{s}}),\dot{\zeta}_{\mathpzc{s}}n_{\mathpzc{s}}\rangle]\der \mathpzc{s}\biggr].
\end{align}
From \eqref{NPIDD.1} and inequality  $\langle\gamma,y\rangle\leq\psi_{\varepsilon}(|\gamma|^{2}-g   (x)^{2})+l_{\varepsilon}(x,y)$, we have 
\begin{align}\label{expand1.0.1}
u^{\varepsilon}(\tilde{x})&\leq\E_{\tilde{x}}\biggl[\expo^{-q[\mathpzc{t}\wedge\tau]} u^{\varepsilon}(X_{\mathpzc{t}\wedge\tau})+\int_{0}^{\tau}\expo^{-q\mathpzc{s}}[ h (X_{\mathpzc{s}})+l_{\varepsilon}(X_{\mathpzc{s}},\dot{\zeta}_{\mathpzc{s}}n_{\mathpzc{s}})]\der \mathpzc{s}\biggr],
\end{align}
since $h+l_{\varepsilon}\geq0$. Observe that 
\begin{align*}
 \E_{\tilde{x}}[\expo^{-q[\mathpzc{t}\wedge\tau]} u^{\varepsilon}(X_{\mathpzc{t}\wedge\tau})]&=\E_{\tilde{x}}[\expo^{-q[\mathpzc{t}\wedge\tau]} u^{\varepsilon}(X_{\mathpzc{t}\wedge\tau})\uno_{\{\tau<\infty\}}]+\E_{\tilde{x}}[\expo^{-q\mathpzc{t}} u^{\varepsilon}(X_{\mathpzc{t}})\uno_{\{\tau=\infty\}}].
\end{align*}
On the event $\{\tau<\infty\}$, we have  $\lim_{t\rightarrow\infty}\expo^{-q[\mathpzc{t}\wedge\tau]} u^{\varepsilon}(X_{\mathpzc{t}\wedge\tau})=\expo^{-q\tau} u^{\varepsilon}(X_{\tau})=0$, since $u^{\varepsilon}=0$ on $\overline{\set}_{\inted}\setminus\set$, and by Lemma \ref{Lb1}, $0\leq \expo^{-q[\mathpzc{t}\wedge\tau]} u^{\varepsilon}(X_{\mathpzc{t}\wedge\tau})\le C_{1}$ for all $t\geq0$. Then, by Dominated Convergence Theorem, we see 
$\E_{\tilde{x}}[\expo^{-q[\mathpzc{t}\wedge\tau]}u^{\varepsilon}(X_{\mathpzc{t}\wedge\tau})\uno_{\{\tau<\infty\}}]\underset{\mathpzc{t}\rightarrow\infty}{\longrightarrow}0$. Now, on $\{\tau=\infty\}$, we observe that $\expo^{-q\mathpzc{t}}\underset{\mathpzc{t}\rightarrow\infty}{\longrightarrow}0$ and $X_{\mathpzc{t}}\in \set$, for all $ \mathpzc{t}>0$. Since $u^{\varepsilon}$ is a bounded continuous function on $\overline{\set}$, we have that $\E_{\tilde{x}}[\expo^{-q\mathpzc{t}} u^{\varepsilon}(X_{\mathpzc{t}})\uno_{\{\tau=\infty\}}]\leq C_{1}\expo^{-q\mathpzc{t}}\underset{\mathpzc{t}\rightarrow\infty}{\longrightarrow}0$. Then, 
\begin{equation}\label{Eqc.2}
\E_{\tilde{x}}[\expo^{-q[\mathpzc{t}\wedge\tau]} u^{\varepsilon}(X_{\mathpzc{t}\wedge\tau})]\underset{\mathpzc{t}\rightarrow\infty}{\longrightarrow}0.
\end{equation}
Therefore, from here and letting $\mathpzc{t}\rightarrow\infty$ in \eqref{expand1.0.1}, it yields $u^{\varepsilon}\leq\mathcal{V}_{n,\zeta}$ on $\overline{\set}$. Let $X^{\varepsilon,*}$ be the solution process to the SDE \eqref{esd3.1.0}, with control $(n^{\varepsilon,*},\zeta^{\varepsilon,*})$ given in \eqref{opt.1}--\eqref{opt.2}. Proceeding in a similar way that in \eqref{exp.1} and noting that the supremum  of $l_{\varepsilon}(x,\eta)$ is attained if $\gamma$ is related to $\eta$ by $\eta=2\psi'_{\varepsilon}(|\gamma|^{2}-g(x)^{2})\gamma$, i.e., 
$$l_{\varepsilon}(x,2\psi'_{\varepsilon}(|\gamma|^{2}- g   (x)^{2})\gamma)=2\psi'_{\varepsilon}(|\gamma|^{2}- g   (x)^{2})|\gamma|^{2}-\psi_{\varepsilon}(|\gamma|^{2}- g   (x)^{2}),$$ 
it follows that
\begin{align}\label{Eqc.1}
u^{\varepsilon}(\tilde{x})&=\E_{\tilde{x}}\biggl[\expo^{-q[\mathpzc{t}\wedge\tau^{*}_{\varepsilon}]} u^{\varepsilon}(X^{\varepsilon,*}_{\mathpzc{t}\wedge\tau^{*}_{\varepsilon}})+\int_{0}^{\mathpzc{t}\wedge\tau^{*}_{\varepsilon}}\expo^{-q\mathpzc{s}} [ h (X^{\varepsilon,*}_{\mathpzc{s}})+l_{\varepsilon}(X^{\varepsilon,*}_{\mathpzc{s}},n^{\varepsilon,*}_{\mathpzc{s}}\dot{\zeta}^{\varepsilon,*}_{\mathpzc{s}})]\der \mathpzc{s}\biggr],
\end{align}
with $\tau^{*}_{\varepsilon}=\inf\{\mathpzc{t}:X^{\varepsilon,*}_{\mathpzc{t}}\notin\set\}$. Notice that 
\begin{equation*}
\int_{0}^{\mathpzc{t}\wedge\tau^{*}_{\varepsilon}}\expo^{-q\mathpzc{s}} [ h (X^{\varepsilon,*}_{\mathpzc{s}})+l_{\varepsilon}(X^{\varepsilon,*}_{\mathpzc{s}},n^{\varepsilon,*}_{\mathpzc{s}}\dot{\zeta}^{\varepsilon,*}_{\mathpzc{s}})]\der \mathpzc{s}\uparrow\int_{0}^{\tau^{*}_{\varepsilon}}\expo^{-q\mathpzc{s}} [ h (X^{\varepsilon,*}_{\mathpzc{s}})+l_{\varepsilon}(X^{\varepsilon,*}_{\mathpzc{s}},n^{\varepsilon,*}_{\mathpzc{s}}\dot{\zeta}^{\varepsilon,*}_{\mathpzc{s}})]\der \mathpzc{s},
\end{equation*}
as $\mathpzc{t}\rightarrow\infty$,  since $h+l_{\varepsilon}\geq0$. Then, by Monotone Convergence Theorem, 
\begin{equation}\label{Eqc.3}
\E_{\tilde{x}}\biggl[\int_{0}^{\mathpzc{t}\wedge\tau^{*}_{\varepsilon}}\expo^{-q\mathpzc{s}} [ h (X^{\varepsilon,*}_{\mathpzc{s}})+l_{\varepsilon}(X^{\varepsilon,*}_{\mathpzc{s}},n^{\varepsilon,*}_{\mathpzc{s}}\dot{\zeta}^{\varepsilon,*}_{\mathpzc{s}})]\der \mathpzc{s}\biggr]\underset{\mathpzc{t}\rightarrow\infty}{\longrightarrow}\mathcal{V}_{n^{\varepsilon,*},\zeta^{\varepsilon,*}}(\tilde{x}).
\end{equation}
Letting $\mathpzc{t}\rightarrow\infty$ in \eqref{Eqc.1} and using \eqref{Eqc.2}, \eqref{Eqc.3}, we conclude 
$u^{\varepsilon}=\mathcal{V}_{n^{\varepsilon,*},\zeta^{\varepsilon,*}}= V^{\varepsilon}$ on $\overline{\set}$.
\end{proof}

To finalize, we present the proof of the main result given in Subsection \ref{Prob.1}.
\begin{proof}[Proof of Proposition \ref{veri1}] 
By Subsection \ref{proofHJB1} and Corollaries \ref{NPIDD.1.2}, \ref{nc1}, we have that there exists a non-increasing sub-sequence $\{u^{\varepsilon_{\kappa}}\}_{\kappa\geq0}$ of $\{u^{\varepsilon}\}_{\varepsilon\in(0,1)}$  such that for each $\kappa\geq0$, $u^{\varepsilon_{\kappa}}$ is the unique non-negative solution to the NPIDD problem \eqref{NPIDD.1.2}, with $\varepsilon=\varepsilon_{\kappa}$, and
\begin{equation*}
u^{\varepsilon_{\kappa}}\underset{\varepsilon_{\kappa}\rightarrow0}{\longrightarrow}u\ \text{in}\ \hol(\overline{\set}_{\inted}),\ \partial_{i}u^{\varepsilon_{\kappa}}\underset{\varepsilon_{\kappa}\rightarrow0}{\longrightarrow}\partial_{i}u\ \text{in}\ \hol_{\loc}(\set),\ \partial_{ij}u^{\varepsilon_{\kappa}}\underset{\varepsilon_{\kappa}\rightarrow0}{\longrightarrow}\partial_{ij}u\  \text{weakly  in}\ \Lp^{p}_{\loc}(\set),
  \end{equation*}
where $p\in(d,\infty)$ is fixed and $u$ is the unique non-negative solution to the HJB equation \eqref{esd5}. Also, from Lemma \ref{convexu1.0}, we know that $u^{\varepsilon_{\kappa}}=\mathcal{V}_{n^{\varepsilon_{\kappa},*},\zeta^{\varepsilon_{\kappa},*}}= V^{\varepsilon_{\kappa}}$ on $\overline{\set}$, with $(n^{\varepsilon_{\kappa},*},\zeta^{\varepsilon_{\kappa},*})$ as in \eqref{opt.1}--\eqref{opt.2}. Notice that $l_{\varepsilon}(x,\beta\gamma)\geq\langle\beta\gamma, g   (x)\gamma\rangle-\psi_{\varepsilon}(| g   (x)\gamma|^{2}-[ g   (x)]^{2})=\beta g   (x)$, with $\beta\in\R$ and $\gamma\in\R^{d}$ a unit vector. Then, from here and considering $X^{\varepsilon_{\kappa},*}$ as in \eqref{esd3.1.0}, it follows that 
\begin{align}\label{NPIDD1.0}
V(\tilde{x})\leq V_{n^{\varepsilon_{\kappa},*},\zeta^{\varepsilon_{\kappa},*}}(\tilde{x})
&=\E_{\tilde{x}}\biggl[\int_{0}^{\tau^{*}_{\varepsilon_{\kappa}}}\expo^{-q\mathpzc{t}}[  h (X^{\varepsilon_{\kappa},*}_{\mathpzc{t}})+\dot{\zeta}^{\varepsilon_{\kappa},*}_{\mathpzc{t}} g   (X^{\varepsilon_{\kappa},*}_{\mathpzc{t}})]\der \mathpzc{t}\biggr]\notag\\
&\leq \E_{\tilde{x}}\biggl[\int_{0}^{\tau^{*}_{\varepsilon_{\kappa}}}\expo^{-q\mathpzc{t}}[  h (X_{\mathpzc{t}}^{\varepsilon_{\kappa},*})+l_{\varepsilon}(X^{\varepsilon_{\kappa},*}_{\mathpzc{t}},\dot{\zeta}^{\varepsilon_{\kappa},*}_{\mathpzc{t}}n^{\varepsilon_{\kappa},*}_{\mathpzc{t}})]\der \mathpzc{t}\biggr]=u^{\varepsilon_{\kappa}}(\tilde{x}),
\end{align}
where $\tau^{*}_{\varepsilon}=\inf\{ \mathpzc{t}>0:X^{\varepsilon,*}_{\mathpzc{t}}\notin\set\}$. Recall that $V_{n^{\varepsilon_{\kappa},*},\zeta^{\varepsilon_{\kappa},*}}$ is the cost function given in \eqref{esd1.1} corresponding to the control $(n^{\varepsilon_{\kappa},*},\zeta^{\varepsilon_{\kappa},*})$, and note that, for this control, the second term in the RHS  of \eqref{esd.1} is zero, since $\zeta^{\varepsilon_{\kappa},*}$  has the continuous part only.  Letting $\varepsilon_{\kappa}\rightarrow0$ in \eqref{NPIDD1.0}, it yields $V\leq u$ on $\overline{\set}$.  Let $X$ be the process that evolves as in  \eqref{esd3.1} and $\tau=\inf\{t>0: X_{\mathpzc{t}}\notin\set\}$, with  $(n,\zeta)\in\mathcal{U}$.  Define $\tau_{m}=$
$\inf\{ \mathpzc{t}>0:X_{\mathpzc{t}}\notin\set_{m}\}$ and $\set_{m}\eqdef\{x\in\set:\dist(x,\partial\set )>1/m\}$, where $m$ is a positive integer  {large enough}. Replacing $\varepsilon$, $\tau$ by $\varepsilon_{\kappa}$, $\tau_{m}$ in \eqref{expand.1}, respectively, using integration by parts and It\^o's formula for $\expo^{-q[\mathpzc{t}\wedge\tau_{m}]}u^{\varepsilon_{\kappa}}(X_{\mathpzc{t}\wedge\tau_{m}})$, it can be verified that  \eqref{expand.1} holds for this case. Notice that  
\begin{align}\label{Ju.1}
\int_{0}^{\mathpzc{t}\wedge\tau_{m}}\expo^{-q\mathpzc{s}}\langle\deri^{1} u^{\varepsilon_{\kappa}}(X_{\mathpzc{s}-}),n_{\mathpzc{s}}\rangle\der\zeta_{\mathpzc{s}}&=\int_{0}^{\mathpzc{t}\wedge\tau_{m}}\expo^{-q\mathpzc{s}}\langle\deri^{1} u^{\varepsilon_{\kappa}}(X_{\mathpzc{s}-}),n_{\mathpzc{s}}\rangle\der\zeta^{\comp}_{\mathpzc{s}}\notag\\
&\quad+\sum_{0\leq\mathpzc{s}\leq\mathpzc{t}\wedge\tau_{m}}\expo^{-q\mathpzc{s}}\langle\deri^{1} u^{\varepsilon_{\kappa}}(X_{\mathpzc{s}-}),n_{\mathpzc{s}}\rangle\Delta\zeta_{\mathpzc{s}},
\end{align}
where $\zeta^{\comp}$ denotes the continuous part of $\zeta$.  {Meanwhile}, since $\Delta X_{\mathpzc{t}}=\Delta  J_{\mathpzc{t}}-n_{\mathpzc{t}}\Delta\zeta_{\mathpzc{t}}$, it can be verified 
\begin{align}\label{Ju.2}
&-\sum_{0\leq\mathpzc{s}\leq\mathpzc{t}\wedge\tau_{m}}\expo^{-q\mathpzc{s}}[u^{\varepsilon_{\kappa}}(X_{\mathpzc{s}-}+\Delta  J_{\mathpzc{s}}-n_{\mathpzc{s}}\Delta\zeta_{\mathpzc{s}})- u^{\varepsilon_{\kappa}}(X_{\mathpzc{s}-}+\Delta J_{\mathpzc{s}})\notag\\
&\qquad\quad\qquad\qquad+u^{\varepsilon_{\kappa}}(X_{\mathpzc{s}-}+\Delta J_{\mathpzc{s}})-u^{\varepsilon_{\kappa}}(X_{\mathpzc{s}-})-\langle\deri^{1}u^{\varepsilon_{\kappa}}(X_{\mathpzc{s}-}),[\Delta  J_{\mathpzc{s}}-n_{\mathpzc{s}}\Delta\zeta_{\mathpzc{s}}]\rangle]\notag\\
&\quad=-\int_{0}^{\mathpzc{t}\wedge\tau_{m}}\int_{\mathcal{S}}\expo^{-q\mathpzc{s}}[u^{\varepsilon_{\kappa}}(X_{\mathpzc{s}-}+z)-u^{\varepsilon_{\kappa}}(X_{\mathpzc{s}-})-\langle\deri^{1}u^{\varepsilon_{\kappa}}(X_{\mathpzc{s}-}),z\rangle]\uno_{\{\rho\in[0,s(X_{\mathpzc{s}},z)]\}}N(\der\rho,\der z,\der\mathpzc{s})\notag\\
&\qquad-\sum_{0\leq\mathpzc{s}\leq\mathpzc{t}\wedge\tau_{m}}\expo^{-q\mathpzc{s}}[\mathcal{D}[u^{\varepsilon_{\kappa}}]_{\mathpzc{s}}+\langle\deri^{1}u^{\varepsilon_{\kappa}}(X_{\mathpzc{s}-}),n_{\mathpzc{s}}\rangle\Delta\zeta_{\mathpzc{s}}]\uno_{\{\Delta\zeta_{\mathpzc{s}}\neq0\}}\notag\\
&\quad=-\widetilde{\mathcal{M}}^{\varepsilon_{\kappa}}_{\mathpzc{t}\wedge\tau_{m}}+\int_{0}^{\mathpzc{t}\wedge\tau_{m}}\expo^{-q\mathpzc{s}}\langle\deri^{1}u^{\varepsilon_{\kappa}}(X_{\mathpzc{s}-}),\der  J_{\mathpzc{s}}\rangle-\int_{0}^{\mathpzc{t}\wedge\tau_{m}}\expo^{-q\mathpzc{s}}\inte u^{\varepsilon_{\kappa}}(X_{\mathpzc{s}})\der\mathpzc{s}\notag\\
&\qquad+\int_{0}^{\mathpzc{t}\wedge\tau_{m}}\int_{ \{|z|\in(0,1)\}}\expo^{-q\mathpzc{s}}\langle\deri^{1}u^{\varepsilon_{\kappa}}(X_{\mathpzc{s}}),z\rangle s(X_{\mathpzc{s}},z)\nu(\der z)\der\mathpzc{s}\notag\\
&\qquad-\sum_{0\leq\mathpzc{s}\leq\mathpzc{t}\wedge\tau_{m}}\expo^{-q\mathpzc{s}}[\mathcal{D}[u^{\varepsilon_{\kappa}}]_{\mathpzc{s}}+\langle\deri^{1}u^{\varepsilon_{\kappa}}(X_{\mathpzc{s}-}),n_{\mathpzc{s}}\rangle\Delta\zeta_{\mathpzc{s}}]\uno_{\{\Delta\zeta_{\mathpzc{s}}\neq0\}},
\end{align}
with $\widetilde{\mathcal{M}}^{\varepsilon_{\kappa}}$ as in \eqref{M.1} and 
$\mathcal{D}[u^{\varepsilon_{\kappa}}]_{\mathpzc{s}}\eqdef [u^{\varepsilon_{\kappa}}(X_{\mathpzc{s}-}+\Delta  J_{\mathpzc{s}}-n_{\mathpzc{s}}\Delta\zeta_{\mathpzc{s}})-u^{\varepsilon_{\kappa}}(X_{\mathpzc{s}-}+\Delta  J_{\mathpzc{s}})]\uno_{\{\Delta\zeta\neq0\}}.$
Applying \eqref{Ju.1}--\eqref{Ju.2} in \eqref{expand.1}, it is easy to verify that 
\begin{align}\label{expand}
  u^{\varepsilon_{\kappa}}(\tilde{x})&=
\expo^{-q[\mathpzc{t}\wedge\tau_{m}]}u^{\varepsilon_{\kappa}}(X_{\mathpzc{t}\wedge\tau_{m}})+\int_{0}^{\mathpzc{t}\wedge\tau_{m}}\expo^{-q\mathpzc{s}}[q-\Gamma_{1}]u^{\varepsilon_{\kappa}}(X_{\mathpzc{s}})\der \mathpzc{s}\notag\\
  &\quad-\mathcal{M}^{\varepsilon_{\kappa}}_{\mathpzc{t}\wedge\tau_{m}}+\int_{0}^{\mathpzc{t}\wedge\tau_{m}}\expo^{-q\mathpzc{s}}\langle\deri^{1} u^{\varepsilon_{\kappa}}(X_{\mathpzc{s}-}),n_{\mathpzc{s}}\rangle\der\zeta_{\mathpzc{s}}^{\comp}-\sum_{0\leq\mathpzc{s}\leq\mathpzc{t}\wedge\tau_{m}}\expo^{-q\mathpzc{s}}\mathcal{D}[u^{\varepsilon_{\kappa}}]_{\mathpzc{s}},
\end{align}
where $q>0$, $\Gamma_{1} $ is as in \eqref{esd3.0},  and  $\mathcal{M}^{\varepsilon_{\kappa}}$ is the square integrable martingale given by \eqref{M.2}. From \eqref{NPIDD.1}, $[q-\Gamma_{1}]u^{\varepsilon_{\kappa}}(X_{\mathpzc{s}})\leq  h (X_{\mathpzc{s}})$, for all $\mathpzc{s}\in[0,\mathpzc{t}\wedge\tau_{m})$. Then, taking expected value in \eqref{expand},
\begin{align}\label{expand1}
u^{\varepsilon_{\kappa}}(\tilde{x})&\leq
\E_{\tilde{x}}\bigg[\expo^{-q[\mathpzc{t}\wedge\tau_{m}]}u^{\varepsilon_{\kappa}}(X_{\mathpzc{t}\wedge\tau_{m}})\notag\\
&\quad+\int_{0}^{\mathpzc{t}\wedge\tau_{m}}\expo^{-q\mathpzc{s}}[ h (X_{\mathpzc{s}})\der \mathpzc{s}+|\deri^{1} u^{\varepsilon_{\kappa}}(X_{\mathpzc{s}})|\der\zeta_{\mathpzc{s}}^{\comp}]-\sum_{0\leq\mathpzc{s}\leq\mathpzc{t}\wedge\tau_{m}}\expo^{-q\mathpzc{s}}\mathcal{D}[u^{\varepsilon_{\kappa}}]_{\mathpzc{s}}\bigg].
\end{align}
Define $g_{1}(\mathpzc{t}\wedge\tau_{m},X_{\mathpzc{t}\wedge\tau_{m}})=\sum_{0\leq\mathpzc{s}\leq\mathpzc{t}\wedge\tau_{m}}\expo^{-q\mathpzc{s}}\mathcal{D}[u]_{\mathpzc{s}}$. Then, letting $\varepsilon_{\kappa}\rightarrow0$ in \eqref{expand1}, by  Dominated Convergence Theorem, and using $u^{\varepsilon_{\kappa}}$,  $|\deri^{1}u^{\varepsilon_{\kappa}}|$ are uniformly bounded by $C_{1}$, $C_{3}$ on $\overline{\set}_{\inted}$ and $\overline{\set}_{m}$, respectively,  $u^{\varepsilon_{\kappa}}(X_{\mathpzc{s}})\underset{\varepsilon_{\kappa}\rightarrow0}{\longrightarrow}u(X_{\mathpzc{s}})$, $|\deri^{1}u^{\varepsilon_{\kappa}}(X_{\mathpzc{s}})|\underset{\varepsilon_{\kappa}\rightarrow0}{\longrightarrow}|\deri^{1}u(X_{\mathpzc{s}})|$, $\mathcal{D}[u^{\varepsilon_{\kappa}}]_{\mathpzc{s}}\underset{\varepsilon_{\kappa}\rightarrow0}{\longrightarrow}\mathcal{D}[u]_{\mathpzc{s}}$ for $\mathpzc{s}\leq \mathpzc{t}\wedge\tau_{m}$, and  $|\deri^{1}u|\leq g   $ on $\set$, it follows that
\begin{align}\label{expand1.1}
u(\tilde{x})&\leq
\E_{\tilde{x}}\bigg[\expo^{-q[\mathpzc{t}\wedge\tau_{m}]}u(X_{\mathpzc{t}\wedge\tau_{m}})+\int_{0}^{\mathpzc{t}\wedge\tau_{m}}\expo^{-q\mathpzc{s}}[ h (X_{\mathpzc{s}})\der \mathpzc{s}+ g   (X_{\mathpzc{s}})\der\zeta_{\mathpzc{s}}^{\comp}]\bigg]-\E_{\tilde{x}}[g_{1}(\mathpzc{t}\wedge\tau_{m},X_{\mathpzc{t}\wedge\tau_{m}})].
\end{align}
By similar arguments used in \eqref{Eqc.2} and \eqref{Eqc.3}, and noting that $\tau_{m}\uparrow\tau$ as $m\rightarrow\infty$, $\Pro_{\tilde{x}}$-a.s., it can be verified that 
\begin{multline}\label{expand1.2}
\lim_{t\rightarrow\infty}\lim_{m\rightarrow\infty}\E_{\tilde{x}}\bigg[\expo^{-q[\mathpzc{t}\wedge\tau_{m}]}u(X_{\mathpzc{t}\wedge\tau_{m}})+\int_{0}^{\mathpzc{t}\wedge\tau_{m}}\expo^{-q\mathpzc{s}}[ h (X_{\mathpzc{s}})\der \mathpzc{s}+ g   (X_{\mathpzc{s}})\der\zeta_{\mathpzc{s}}^{\comp}]\bigg]\\
=\E_{\tilde{x}}\bigg[\int_{0}^{\tau}\expo^{-q\mathpzc{s}}[ h (X_{\mathpzc{s}})\der \mathpzc{s}+ g   (X_{\mathpzc{s}})\der\zeta_{\mathpzc{s}}^{\comp}]\bigg].
\end{multline}
On the event $\{\tau=\infty\}$, we have that for each $\mathpzc{s}>0$ such that $\Delta \zeta _{\mathpzc{s}}\neq0$,   $X_{\mathpzc{s}-}+\Delta  J_{\mathpzc{s}}-n_{\mathpzc{s}}\Delta\zeta_{\mathpzc{s}}\in\set$.  {Meanwhile},  $X_{\mathpzc{s}-}+\Delta  J_{\mathpzc{s}}\in \set$ or $X_{\mathpzc{s}-}+\Delta  J_{\mathpzc{s}}\in\set_{\inted}\setminus\set$. If $X_{\mathpzc{s}-}+\Delta  J_{\mathpzc{s}}\in \set$, by Mean Value Theorem, 
\begin{align}\label{esd14}
-\mathcal{D}[u]_{\mathpzc{s}}&\leq\Delta\zeta_{\mathpzc{s}}\int_{0}^{1}|\deri^{1} u( X_{\mathpzc{s}-}+\Delta  J_{\mathpzc{s}}-\lambda n_{\mathpzc{s}}\Delta\zeta_{\mathpzc{s}})|\der\lambda\leq\Delta\zeta_{\mathpzc{s}}\int_{0}^{1} g   ( X_{\mathpzc{s}-}+\Delta  J_{\mathpzc{s}}-\lambda  n_{\mathpzc{s}}\Delta\zeta_{\mathpzc{s}})\der\lambda,
\end{align}
since $|\deri^{1}u|\leq  g$ in $\set$. If $X_{\mathpzc{s}-}+\Delta  J_{\mathpzc{s}}\in\set_{\inted}\setminus\set$,  { we have that the line segment between  $X_{\mathpzc{s}-}+\Delta  J_{\mathpzc{s}}$ and $X_{\mathpzc{s}-}+\Delta  J_{\mathpzc{s}}-n_{\mathpzc{s}}\Delta\zeta_{\mathpzc{s}}$, that is described by $X_{\mathpzc{s}-}+\Delta J_{\mathpzc{s}}-\lambda n_{\mathpzc{s}}\Delta \zeta_{\mathpzc{s}}$, with $\lambda\in[0,1]$,  intersects $\partial\set$ in a unique point $X^{*}_{\mathpzc{s}}\eqdef X_{\mathpzc{s}-}+\Delta J_{\mathpzc{s}}-\lambda^{*} n_{\mathpzc{s}}\Delta \zeta_{\mathpzc{s}}$, for some $\lambda^{*}\in(0,1)$, since $\set$ is convex.  Then, noting that $X_{\mathpzc{s}-}+\Delta  J_{\mathpzc{s}}-n_{\mathpzc{s}}\Delta\zeta_{\mathpzc{s}}-X^{*}_{\mathpzc{s}}=-[1-\lambda^{*}]n_{\mathpzc{s}}\Delta\zeta_{\mathpzc{s}}$, and using again Mean Value Theorem and the fact that $u(X^{*}_{\mathpzc{s}})=0$,} 
\begin{align}\label{expand9.1}
{-u(X_{\mathpzc{s}-}+\Delta  J_{\mathpzc{s}}-n_{\mathpzc{s}}\Delta\zeta_{\mathpzc{s}})}&={-[u(X_{\mathpzc{s}-}+\Delta  J_{\mathpzc{s}}-n_{\mathpzc{s}}\Delta\zeta_{\mathpzc{s}})-u(X^{*}_{\mathpzc{s}})]}\notag\\
&\leq[1-\lambda^{*}]\Delta\zeta_{\mathpzc{s}}\int_{0}^{1} g   (X^{*}_{\mathpzc{s}}-\lambda[1-\lambda^ {*}]n_{\mathpzc{s}}\Delta\zeta_{\mathpzc{s}})\der\lambda.
\end{align}
Meanwhile, observe that
\begin{align*}
{\int_{\lambda^{*}}^{1}g(X_{\mathpzc{s}-}+\Delta  J_{\mathpzc{s}}-\lambda  n_{\mathpzc{s}}\Delta\zeta_{\mathpzc{s}})\der\lambda=[1-\lambda^{*}]\int_{0}^{1}g(X^{*}_{\mathpzc{s}}-\lambda[1-\lambda^ {*}]n_{\mathpzc{s}}\Delta\zeta_{\mathpzc{s}})\der\lambda.}
\end{align*}
Then, from here,  it is easy to verify 
\begin{equation}\label{expand9.2}
[1-\lambda^{*}]\int_{0}^{1}g(X^{*}_{\mathpzc{s}}-\lambda[1-\lambda^ {*}]n_{\mathpzc{s}}\Delta\zeta_{\mathpzc{s}})\der\lambda\leq \int_{0}^{1} g   ( X_{\mathpzc{s}-}+\Delta  J_{\mathpzc{s}}-\lambda  n_{\mathpzc{s}}\Delta\zeta_{\mathpzc{s}})\der\lambda.
\end{equation}
Therefore, by \eqref{expand9.1}--\eqref{expand9.2} and that $u(X_{\mathpzc{s}-}+\Delta  J_{\mathpzc{s}})=0$, it yields \eqref{esd14}. From here and by Monotone Convergence Theorem, we have
\begin{align}\label{ineq.1.1}
\lim_{t\rightarrow\infty}&\lim_{m\rightarrow\infty}\E_{\tilde{x}}[-g_{1}(\mathpzc{t}\wedge\tau_{m},X_{\mathpzc{t}\wedge\tau_{m}})\uno_{\{\tau=\infty\}}]\notag\\
 &\leq\lim_{t\rightarrow\infty}\lim_{m\rightarrow\infty}\E_{\tilde{x}}\biggr[\uno_{\{\tau=\infty\}}\sum_{0\leq\mathpzc{s}\leq\mathpzc{t}\wedge\tau_{m}}\expo^{-q\mathpzc{s}}\Delta\zeta_{\mathpzc{s}}\int_{0}^{1} g   ( X_{\mathpzc{s}-}+\Delta  J_{\mathpzc{s}}-\lambda n_{\mathpzc{s}}\Delta\zeta_{\mathpzc{s}})\der\lambda\biggl]\notag\\
 &\leq\E_{\tilde{x}}\biggr[\uno_{\{\tau=\infty\}}\sum_{\mathpzc{s}\geq0}\expo^{-q\mathpzc{s}}\Delta\zeta_{\mathpzc{s}}\int_{0}^{1} g   ( X_{\mathpzc{s}-}+\Delta  J_{\mathpzc{s}}-\lambda n_{\mathpzc{s}}\Delta\zeta_{\mathpzc{s}})\der\lambda\biggl].
\end{align}
Now,  on $\{\tau<\infty\}$, {for $\mathpzc{s}<\tau$ we can use the same arguments as for the case $\{\tau=\infty\}$; for $\mathpzc{s}=\tau$ we have $X_{\tau-}+\Delta  J_{\tau}- n_{\tau}\Delta\zeta_{\tau}\in\overline{\set}_{\inted}\setminus\set$ and either $X_{\tau-}+\Delta  J_{\tau}\in\set$ or $X_{\tau-}+\Delta  J_{\tau}\in\overline{\set}_{\inted}\setminus\set$, then similar arguments as before apply}. Then,  {by using the} Monotone Convergence Theorem, we have 
\begin{multline}\label{expand10}
\lim_{t\rightarrow\infty}\lim_{m\rightarrow\infty}\E_{\tilde{x}}[-g_{1}(\mathpzc{t}\wedge\tau_{m},X_{\mathpzc{t}\wedge\tau_{m}})\uno_{\{\tau<\infty\}}]\\
\leq\E_{\tilde{x}}\biggr[\uno_{\{\tau<\infty\}}\sum_{0\leq\mathpzc{s}\leq\tau}\expo^{-q\mathpzc{s}}\Delta\zeta_{\mathpzc{s}}\int_{0}^{1} g   ( X_{\mathpzc{s}-}+\Delta  J_{\mathpzc{s}}-\lambda n_{\mathpzc{s}}\Delta\zeta_{\mathpzc{s}})\der\lambda\biggl].
\end{multline}
Therefore, letting $m\rightarrow\infty$ and $t\rightarrow\infty$ in \eqref{expand1.1}, and using \eqref{esd.1}, \eqref{expand1.2}, \eqref{ineq.1.1} and \eqref{expand10}, it yields $u\leq V_{n,\zeta}$ on $\overline{\set}$. From here and \eqref{vf1},  $u\leq V$ on $\overline{\set}$. By the  arguments seen previously, we conclude that $u=V$ on $\overline{\set}$.
\end{proof}

\subsection{About penalized optimal controls}
As  {discussed} previously, the value function $V$, given in \eqref{vf1}, satisfies the HJB \eqref{esd5}. This means that the domain set $\set$ is divided into two parts. The first part, {defined as} $\mathcal{E}\subseteq\set$, is where $V$  satisfies the elliptic integro-differential equation $[q-\Gamma_{1}] V=h$, which suggests that the optimal control corresponding to this problem will not be exercised on $\mathcal{E}$. Otherwise the `optimal control' will exercise a force and a direction  at $x\in\set\setminus\mathcal{E}$  in such a way that the process $ X^{n,\zeta}_{\mathpzc{t}}$ will be pushed back to some point $y \in\partial\mathcal{E}$.

To construct an optimal strategy to the problem \eqref{vf1}, it is  {necessary} to verify that $\partial\mathcal{E}$ is at least of class $\hol^{1}$, which is not easy to get; this  is  {currently}  the topic of a work in progress by the authors.  In the literature, we can find problems of this type  {that have been} successfully solved in some cases; see, e.g., \cite{app,A1999,BKY2014,DZ1998,GZ2015,JJZ2008,kruk,soner}.   

Another way to address the problem \eqref{vf1} is by means of $\varepsilon$-penalized optimal controls which have been  constructed in \eqref{opt.1}--\eqref{opt.2}. From Lemma \ref{convexu1.0} and proof of Lemma \ref{veri1}, we know that $V^{\varepsilon_{\kappa}}\downarrow V$ as $\varepsilon_{\kappa}\downarrow0$ on $\set$, and 
$V\leq V_{n^{\varepsilon_{\kappa},*},\zeta^{\varepsilon_{\kappa},*}}\leq V^{\varepsilon_{\kappa}}\ \text{on $\overline{\set}$}$,
 where $(n^{\varepsilon_{\kappa},*},\zeta^{\varepsilon_{\kappa},*})$ as in \eqref{opt.1}--\eqref{opt.2}, and  $V_{n^{\varepsilon_{\kappa},*},\zeta^{\varepsilon_{\kappa},*}},V, V^{\varepsilon_{\kappa}}$ are given by \eqref{esd1.1}, \eqref{vf1} and \eqref{Vfp1}, respectively. Taking $\varepsilon_{\kappa}$ small enough, we have that the control $(n^{\varepsilon_{\kappa},*},\zeta^{\varepsilon_{\kappa},*})$ is exercised as follows:  if the controlled process $X^{ \varepsilon_{\kappa},*}$ satisfies $|\deri^{1}V^{\varepsilon_{\kappa}}(X^{\varepsilon_{\kappa},*}_{\mathpzc{t}})|\leq g(X^{\varepsilon_{\kappa},*}_{\mathpzc{t}})$ with $\mathpzc{t}\in[0,\tau^{*}_{\varepsilon_{\kappa}}]$ and $\tau^{*}_{\varepsilon_{\kappa}}=\inf\{\mathpzc{t}>0:X^{\varepsilon_{\kappa},*}_{\mathpzc{t}}\notin\set\}$,  then $\zeta^{\varepsilon_{\kappa},*}_{\mathpzc{t}}\equiv0$ and $X^{\varepsilon_{\kappa},*}_{\mathpzc{t}}$ will stay in $\mathcal{E}$. If $0<|\deri^{1}V^{\varepsilon_{\kappa}}(X^{\varepsilon_{\kappa},*}_{\mathpzc{t}})|^{2}-g(X^{\varepsilon_{\kappa},*}_{\mathpzc{t}})^{2}<2\varepsilon_{\kappa}$ with $\mathpzc{t}\in[0,\tau^{*}_{\varepsilon_{\kappa}}]$, the process $X^{\varepsilon_{\kappa},*}_{\mathpzc{t}}$ will be crossing  $\partial\mathcal{E}$ persistently. Otherwise, $(n^{\varepsilon_{\kappa},*}_{\mathpzc{t}},\zeta^{\varepsilon_{\kappa},*}_{\mathpzc{t}})$   will exercise  a force $\frac{2}{\varepsilon}|\deri^{1}V^{\varepsilon_{\kappa}}(X^{\varepsilon_{\kappa},*}_{\mathpzc{t}})|$ and a direction $-\frac{\deri^{1}V^{\varepsilon_{\kappa}}(X^{\varepsilon_{\kappa},*}_{\mathpzc{t}})}{|\deri^{1}V^{\varepsilon_{\kappa}}(X^{\varepsilon_{\kappa},*}_{\mathpzc{t}})|}$ at $X^{\varepsilon_{\kappa},*}_{\mathpzc{t}}$ in such a way that it will be pushed back to  $\partial\mathcal{E}$.  

To finalize this section,  Zhu \cite{zhu} also solved a similar  problem by means of $\varepsilon$-penalized optimal controls   when the state process is a multidimensional diffusion process.

\section{Conclusions  and some further work}
In this paper   we have guaranteed, under Assumptions (A1)--(A4),  the existence and uniqueness for the strong (in the a.e. sense) and classical solutions to the HJB and NPIDD equations presented in \eqref{p1} and \eqref{p13.0}, respectively. It should be noted that one of main contributions of this work is Assumption (A4),  {which} permits the L\'evy measure $\nu$ to be  infinite on $\R^{d}_{*}$. This assumption also played an important role in the proofs of Lemmas \ref{lemfrontera1.0} and \ref{cotaphi}. 

Another main result achieved in this paper   is  {the establishment of a }  the strong relationship between the value functions $V$, $V^{\varepsilon}$  given in \eqref{vf1}, \eqref{Vfp1},  and the solutions $u$, $u^{\varepsilon}$ to the equations \eqref{esd5}, \eqref{NPIDD.1}, respectively. Although, the optimal control process for the singular stochastic problem   \eqref{vf1} was not given, and this is still an open problem, we constructed a family of $\varepsilon$-optimal  absolutely continuous control processes $\{(n^{\varepsilon_{\kappa},*},\zeta^{\varepsilon_{\kappa},*})\}_{\kappa\geq1}$; see \eqref{opt.1}--\eqref{opt.2}, such that the limit of their value functions $V^{\varepsilon_{\kappa}}$ (as $\varepsilon_{\kappa}\rightarrow0$) agrees with the value function $V$.

There are some extensions to be considered and directions for future research:

\begin{enumerate}[(i)]
	\item One {of the natural extensions of this work would be} to study the HJB and NPIDD equations \eqref{p1} and \eqref{p13.0}, respectively, when the integral operator $\inted w$ has the form $\int_{\R^{d}_{*}}[w(\cdot+z)-w-\langle\deri^{1}w,z\rangle\uno_{\{|z|\in(0,1)\}}]s(\cdot,z)\nu(\der z),$ and the L\'evy measure $\nu$ has unbounded variation, i.e., $\int_{\R^{d}_{*}}[|z|^{2}\wedge1]\nu(\der z)<\infty$. In this case, the main difficulty {lies in obtaining} results  similar to Lemmas \ref{lemfrontera1.0} and \ref{cotaphi}  {because} we must have an \textit{\`a priori} estimate  of $\int_{\{|z|\in(0,1)\}}\big[\int_{0}^{1} |\deri^{2}u^{\varepsilon}(\cdot+tz)|\der t\big]|z|^{2}s(\cdot,z)\nu(\der z)$ independent of $\varepsilon$.
	
	\item Another extension is to generalize the gradient constraint that appears in \eqref{p1}, i.e., to study the HJB equation presented  in works as \cite{hynd2} or \cite{Hynd3}, when the operator is a partial integro-differential operator as in \eqref{p6.1}.  
	
	\item In parallel to this research,  the stochastic control problems in different branches of applied probability (insurances, inventories, etc.),   {which are} closely related to these  HJB equations,  {may} be analyzed.  
\end{enumerate}

\subsection*{Acknowledgement}
The authors would like to thank  the anonymous reviewers for their comments and suggestions, which improved the quality of this paper.

\end{document}